\documentclass[10pt]{amsart}

\usepackage{amsmath,amssymb,amsthm,bm}
\usepackage{enumitem}
\usepackage{verbatim}
\usepackage[utf8]{inputenc}
\usepackage{tikz}
\usetikzlibrary{matrix,arrows}

\usepackage{mathrsfs}

\allowdisplaybreaks

\usepackage{setspace}
\theoremstyle{plain}
\newtheorem{theorem}{Theorem}
\newtheorem*{theoremnoname}{Theorem}
\newtheorem{proposition}[theorem]{Proposition}
\newtheorem{lemma}[theorem]{Lemma}
\newtheorem{corollary}[theorem]{Corollary}

\theoremstyle{definition}
\newtheorem{definition}[theorem]{Definition}
\newtheorem{example}[theorem]{Example}
\newtheorem{remark}[theorem]{Remark}
\newtheorem*{notation}{Notation}

\renewcommand{\a}{\mathbf{a}}
\newcommand{\A}{\mathcal{A}}
\renewcommand{\b}{\mathbf{b}}

\renewcommand{\c}{\mathbf{c}}
\newcommand{\bH}{{\mathbf{H}}}
\newcommand{\bHbar}{{\overline{\bH}}}

\newcommand{\eps}{\boldsymbol\varepsilon}
\newcommand{\f}{\mathbf{f}}
\newcommand{\fbar}{\bar{f}}
\newcommand{\g}{\mathbf{g}}
\newcommand{\gbar}{\bar{g}}

\newcommand{\J}{\mathrel{\mathscr{J}}}

\newcommand{\leqH}{\leq_{\sH}}
\newcommand{\leqL}{\leq_{\sL}}
\newcommand{\leqR}{\leq_{\sR}}

\renewcommand{\L}{\mathrel{\mathscr{L}}}
\newcommand{\onto}{\twoheadrightarrow}

\renewcommand{\phi}{\varphi}
\newcommand{\PL}{\mathrm{PL}}
\newcommand{\q}{\vec{q}}
\newcommand{\relto}{\mathrel{\ooalign{\hfil$\mapstochar\mkern5mu$\hfil\cr$\to$\cr}}}
\newcommand{\R}{\mathrel{\mathscr{R}}}
\newcommand{\sF}{\mathcal{F}}
\newcommand{\sFbar}{\overline{\sF}}
\newcommand{\sH}{\mathrel{\mathscr{H}}}
\newcommand{\sL}{\L}
\newcommand\sR{\R}
\newcommand{\Sat}{\mathrm{Sat}}
\newcommand{\St}{\mathrm{St}}
\newcommand{\SH}{S^\bH}
\newcommand{\SHbar}{\overline{S^\bH}}
\newcommand{\taubar}{\overline{\tau}}
\renewcommand{\th}{\mathrm{th}}
\renewcommand{\theta}{\vartheta}
\newcommand{\T}{\mathcal{T}}

\let\accentu\u     
\renewcommand{\u}{\mathbf{u}}
\newcommand{\ubar}{\bar{u}}
\renewcommand{\vec}{\mathbf}
\newcommand{\V}{\mathbf{V}}

\newcommand{\X}{\mathscr{X}}

\numberwithin{equation}{section}
\numberwithin{theorem}{section}

\title{Pointlike sets for varieties determined by groups}
\author{S.~J.~v. Gool and B.~Steinberg}
\address{Department of Mathematics\\    City College of New York\\Convent Avenue at 138th Street\\New York, New York 10031\\USA}
\email{samvangool@me.com}
\email{bsteinberg@ccny.cuny.edu}
\thanks{The first-named author was supported by the European Union's Horizon 2020 research and innovation programme under the Marie Sklodowska-Curie grant \#655941; the second-named author was supported by NSA MSP \#H98230-16-1-0047.}

\date{January 14, 2018}

\begin{document}

\begin{abstract}
For a variety of finite groups $\bH$, let $\bHbar$ denote the variety of finite semigroups all of whose subgroups lie in $\bH$. We give a characterization of the subsets of a finite semigroup that are pointlike with respect to $\bHbar$. Our characterization is effective whenever $\bH$ has a decidable membership problem. In particular, the separation problem for $\bHbar$-languages is decidable for any decidable variety of finite  groups $\bH$.  This generalizes Henckell's theorem on decidability of aperiodic pointlikes.
\end{abstract}

\maketitle
\section{Introduction}
Motivated by Krohn-Rhodes decomposition theory~\cite{KR65,KR68} and formal language theory, Eilenberg~\cite{Eilenberg76} suggested the notion of a variety of finite semigroups (a class of finite semigroups closed under formation of finite direct products, subsemigroups and homomorphic images) as the fundamental organizing principle in finite semigroup theory and the algebraic theory of automata.  Moreover, he and Tilson emphasized the importance of relational morphisms (generalizing divisions in the sense of Krohn and Rhodes) as the appropriate arrows to consider between finite semigroups.  A relational morphism is essentially a multi-valued semigroup homomorphism: a relation $\varphi\colon S\relto T$ is a relational morphism if $s\varphi\neq \emptyset$  and $s\varphi s'\varphi\subseteq (ss')\varphi$ for all $s,s'\in S$.   A modern treatise on finite semigroup theory from the varietal viewpoint is the book of Rhodes and the second author~\cite{RS2009}.

A subset $X$ of a finite semigroup $S$ is $\mathbf V$-pointlike (with respect to a variety $\V$) if it relates to a point under all relational morphisms from $S$ into $\mathbf V$, that is, for all relational morphisms $\varphi\colon S\relto T$ with $T\in \mathbf V$, there exists $t\in T$ with $X\subseteq t\varphi^{-1}$. There is also a nice profinite reformulation of the notion (cf.~\cite{RS2009}). If $S$ is a finite semigroup generated by a finite set $A$, $\gamma\colon \widehat{A^+}\to S$ is the canonical projection from the free profinite semigroup on $A$ to $S$ and $\pi_{\mathbf V}\colon \widehat{A^+}\to \widehat{F}_{\mathbf V}(A)$ is the canonical projection to the free pro-$\mathbf V$ semigroup, then $X\subseteq S$ is $\mathbf V$-pointlike if and only if $X\subseteq t\pi_{\mathbf V}^{-1} \gamma$ for some $t\in \widehat{F}_{\mathbf V}(A)$.

A variety $\mathbf V$ is said to have decidable pointlikes if there is an algorithm which takes as input a finite semigroup and one of its subsets and outputs whether that subset is $\mathbf V$-pointlike.  It is easy to see that $S\in \mathbf V$ if and only if the $\mathbf V$-pointlike subsets of $S$ are the singletons.  Thus the $\mathbf V$-pointlikes problem is at least as hard as the membership problem for $\mathbf V$.  Rhodes and the second author~\cite{RS99} produced the first example of a variety with decidable membership but undecidable pointlikes.  Later, Auinger and the second author~\cite{AS03} gave an example of a variety of finite metabelian groups with decidable membership problem but undecidable pointlikes.  In both constructions, it is undecidable if a pair of elements is pointlikes.  It is obvious that pointlike sets are decidable for any locally finite variety with computable free objects.

Pointlike sets were considered by Rhodes and his students in Berkeley since the 1970s.  The first major result in the subject was Henckell's Theorem~\cite{Hen1988} that the variety of aperiodic semigroups has decidable pointlikes.  The initial motivation for this problem was in an attempt to prove the decidability of Krohn-Rhodes complexity~\cite{KR68}.  However, it is just a first step in the problem.

The next major result was part of Ash's solution to the Rhodes Type II conjecture~\cite{Ash91}.  Ash showed that pointlikes were decidable for the variety $\mathbf G$ of all finite groups.  Henckell and Rhodes had conjectured~\cite{HR91} a description of the $\mathbf G$-pointlikes that was subsequently proved by Ash~\cite{Ash91}.  Henckell and Rhodes had shown that their conjecture implied the famous equality $\mathbf{PG}=\mathbf{BG}$ of the variety $\mathbf{PG}$ generated by power sets of finite groups and the variety of block groups $\mathbf{BG}$ (finite semigroups in which regular elements have a unique inverse);  see~\cite{HMPR} for details.

As the previous example shows, sometimes decidability of $\mathbf V$-pointlikes allows one to decide membership in related varieties.  For example, the second author observed in 1996 that if $\mathbf V$ has decidable pointlikes and $\mathbf W$ is locally finite with computable free objects (e.g., if $\mathbf W$ is generated by a finite semigroup), then the membership problem for the join $\mathbf V\vee \mathbf W$ is decidable~\cite{Slice}; see also~\cite{Almeida99}.  He also used Ash's result on the decidability of $\mathbf G$-pointlikes to show that $\mathbf J\vee \mathbf G $ is decidable where $\mathbf J$ is the variety of finite $\J$-trivial semigroups~\cite{Slice,Slice2}; see also~\cite{Jvg}. Henckell showed~\cite{Henckell04,HRS2010} that if $\mathbf W$ has decidable pointlikes and the Mal'cev product $\mathbf A\mbox{\textcircled{\footnotesize{$m$}}} \mathbf W$ coincides with $\mathbf W$, then the Mal'cev product $\mathbf V\mbox{\textcircled{\footnotesize{$m$}}} \mathbf W$ has decidable membership for any decidable variety $\mathbf V$.

There has been a great deal of work on pointlike sets and there are many decidability results. A by no means exhaustive list of examples include: $\mathbf J$~\cite{Slice,Jhyp}, $\mathbf R$ (the variety of $\mathscr R$-trivial semigroups)~\cite{ACZ08}, any decidable variety of finite abelian groups~\cite{DMS07}, the variety of finite $p$-groups for a prime $p$~\cite{St01} and the variety of finite nilpotent groups~\cite{K17}.  Also if $\mathbf V$ is a variety of finite monoids with decidable pointlikes, then the second author showed that $\mathbf V\ast \mathbf D$ has decidable pointlikes, where $\ast$ is the semidirect product of varieties and $\mathbf D$ is the variety of finite semigroups whose idempotents are right zeroes.  Varieties of the form $\mathbf V\ast \mathbf D$ form a large class of important varieties from the point-of-view of language theory due to the work of Straubing~\cite{Straubingdelay} and Tilson~\cite{Tilson}. From this result it follows that the varieties $\mathbf{LG}$ of local groups and $\mathbf{LSl}$ of local semilattices have decidable pointlikes.

If $\mathbf H$ is a variety of finite groups, then $\bHbar$ is the variety of finite semigroups whose subgroups belong to $\bH$.  For example, if $\mathbf 1$ denotes the trivial variety of groups, then $\overline{\mathbf 1}$ is the variety of aperiodic semigroups. At the other extreme, $\overline{\mathbf G}$ is the variety of all finite semigroups.  The variety $\overline{\mathbf G_{\mathrm{sol}}}$ (where $\mathbf G_{\mathrm{sol}}$ is the variety of finite solvable groups) plays an important role in circuit complexity (thanks to Barrington's theorem, cf.~\cite{Str1994}).  The  $\overline{\mathbf G_{\mathrm{sol}}}$-recognizable languages are exactly the languages definable in first order logic with modular quantifiers~\cite{Str1994}.

Henckell's paper~\cite{Hen1988} on the decidability of aperiodic pointlikes is not easy.  In~\cite{HRS2010AP}, Henckell, Rhodes and the second author gave a simplified approach to the most difficult step of Henckell's theorem (proving that a certain finite semigroup is aperiodic).  In the process they extended the result to show that if $\pi$ is a recursive set of prime and $\mathbf G_{\pi}$ is the variety of groups whose order is divisible only by primes in $\pi$, then $\overline{\mathbf G_{\pi}}$ has decidable pointlikes (the special case $\pi=\emptyset$ recovers Henckell's result).

Recently, there has been a revived interest in pointlikes sets because of a language theoretic interpretation, due to Almeida~\cite{Almeida99}.  The case of particular interest is that of pointlike pairs.  Let $L_1,L_2\subseteq A^+$ be disjoint regular languages.  Then $L\subseteq A^+$ is called a separator for these languages if $L_1\subseteq L$ and $L_2\subseteq A^+\setminus L$.  The separation problem for $\mathbf V$ asks for an algorithm to decide whether two disjoint regular languages admit a separator that is $\mathbf V$-recognizable.  Almeida's results~\cite{Almeida99} imply that the $\mathbf V$-separation problem is equivalent to the decidability of $\mathbf V$-pointlike pairs.  Since the aperiodic recognizable languages are just the first order definable languages by the Sch\"utzenberger-McNaughton-Pappert Theorem~\cite{Str1994}, Henckell's theorem provides the decidability of the separation problem for first order logic.  Place and Zeitoun~\cite{PlaZei2016FO} recently gave a new proof of Henckell's theorem using the language theoretic reformulation and that the aperiodic recognizable languages are the first order definable languages.  Key to their proof was a very ingenious induction scheme.  The authors~\cite{GooSte2017merge} have recently given a much shorter proof of the decidability of aperiodic pointlikes  via an algebraic approach along the lines of the inductive scheme of Place and Zeitoun.

Our main theorem here is the following sweeping generalization of Henckell's theorem and the results of~\cite{HRS2010AP}.

\begin{theoremnoname}
Let $\mathbf H$ be a variety of finite groups.  Then the following are equivalent.
\begin{enumerate}
  \item $\mathbf H$ has decidable membership.
  \item $\bHbar$ has decidable membership.
  \item The separation problem for $\bHbar$ is decidable.
  \item Pointlikes are decidable for $\bHbar$.
\end{enumerate}
\end{theoremnoname}

Note that $\mathbf A\mbox{\textcircled{\footnotesize{$m$}}} \bHbar=\bHbar$ and so it follows that the Mal'cev product $\mathbf V\mbox{\textcircled{\footnotesize{$m$}}} \bHbar$ is decidable for any decidable variety of finite semigroups $\mathbf V$ and of finite groups $\bH$.

The basic idea of the proof of the theorem is to modify the version of Henckell's construction used in~\cite{HRS2010AP}.  Here we have to work much more seriously with Sch\"utzenberger groups than in~\cite{HRS2010AP} where the varieties of groups in question are defined by pseudoidentities in a single variable.  In particular, $\mathbf H$-kernels of Sch\"utz\-en\-ber\-ger groups play a fundamental role in this paper.  The rough intuition in~\cite{HRS2010AP} is that to show that a transformation semigroup is aperiodic, one has to show that no element has an orbit which is a cycle of prime length.  In this paper, we have to take into account orbits of $\mathbf H$-kernels of subgroups.  We found a crucial simplification of a key argument of~\cite{HRS2010AP} that allows us to do this.

\subsection*{Organization of the paper}
After a section of preliminaries, we state our main result and its corollaries.  In particular, we provide a constructible collection of $\bHbar$-pointlike sets that we aim to prove are precisely the $\bHbar$-pointlikes.  The following section constructs a relational morphism from our initial finite semigroup to another finite semigroup such that the preimage of each point under the relation belongs to the constructible collection of $\bHbar$-pointlikes.  The most technical part of the paper is the proof that the relational morphism has codomain in $\bHbar$, which is proved in the following section.  After a brief section completing the proof of the main theorem, there is a final section discussing alternative characterizations of the $\bHbar$-pointlikes that may be more efficient from the point of view of complexity and which can be used to establish reducibility results in the sense of~\cite{AS2000}.  

\section{Preliminaries}\label{sec:prelim}
\subsection*{Basics}     A \emph{variety of finite (semi)groups} is a class of finite (semi)groups closed under finite direct products, sub(semi)groups and homomorphic images. Throughout the paper, we fix a variety of finite groups $\bH$, and we denote by $\bHbar$ the variety of finite semigroups $S$ such that all subgroups of $S$ belong to $\bH$.

The free semigroup (monoid) on a set $A$ is denoted by $A^+$ ($A^*$).

\begin{notation}
Following \cite{RS2009,HRS2010AP}, we use the notational convention that functions are applied on the right, and hence compositions of functions are read from left to right.
\end{notation}

Let $S$ be a semigroup. We denote by $S^I$ the monoid obtained by adjoining a new identity element $I$ to $S$.
If $s$ is an element of $S^I$, then we denote by $r_s \colon S^I \to S^I$ the function defined by $xr_s := xs$ for $x$ in $S$. The assignment $s \mapsto r_s$ is a semigroup embedding of $S$ into the full transformation semigroup on $S^I$, called the \emph{right regular representation} of $S$.  Note that $S$ is invariant under $r_s$ for $s\in S^I$.


The multiplication of $S$ extends to subsets of $S$, by defining, for any subsets $X$, $Y$ of $S$, $XY := \{xy \mid x \in X, y \in Y\}$. When, e.g., $Y = \{y\}$, we write $Xy$ for $X\{y\}$. The set of subsets of $S$ forms a semigroup $2^S$, which is called the \emph{power semigroup} of $S$, and is partially ordered by inclusion.

\emph{Green's preorders} on $S$ are defined by writing, for $x, y \in S$, $x \leqL y$ when $x \in S^Iy$; $x \leqR y$ when $x \in yS^I$; $x \leq_{\J} y$ when $x \in S^IyS^I$; and $x \leqH y$ when both $x \leqL y$ and $x \leqR y$.
The  \emph{Green equivalence relations} are defined by $x \mathrel{\X} y$ when $x \leq_{\X} y$ and $y \leq_{\X} x$, for any $\X \in \{{\L},{\R},{\J},{\sH}\}$. For any $x \in S$, $L_x$, $R_x$, $J_x$ and $H_x$ denote the ${\L}$-, ${\R}$-, ${\J}$-, and ${\sH}$-class of $x$, respectively.
The \emph{strict} partial orders are defined by $x <_{\X} y$ when $x \leq_{\X} y$, but it is not the case that $x \mathrel{\X} y$, where $\X \in \{{\L},{\R},{\J},{\sH}\}$. %
Finite semigroups are \emph{stable}, which means that, whenever $x \mathrel{\J} y$, we have $x \geq_{\R} y$ implies $x \R y$ and $x \geq_{\L} y$ implies $x \L y$, cf. e.g., \cite[Thm.~A.2.4]{RS2009}.

A \emph{relational morphism} $\phi \colon S \relto T$ is a function $\phi \colon S \to 2^T$ such that (i) $s\phi \neq \emptyset$ for every $s$ in $S$, (ii) $s_1\phi s_2\phi \subseteq (s_1s_2)\phi$ for every $s_1, s_2$ in $S$. In other words, the \emph{graph} of $\phi$, $\#\phi := \{(s,t) \mid s \in S, t \in s\phi\}$, is a subsemigroup of $S \times T$ such that the restriction of the projection onto the first coordinate is surjective. We denote by $p_S \colon \#\phi \to S$ and $p_T \colon \#\phi \to T$ the restrictions of the projection maps.

\subsection*{Pointlike sets} Let $\V$ be a variety of finite semigroups. A subset $X$ of a finite semigroup $S$ is called \emph{$\V$-pointlike}  if, for any relational morphism $\phi \colon S \relto T$ with $T \in \V$, there exists $t \in T$ such that $X \subseteq t\phi^{-1}$. Any singleton set is $\V$-pointlike, and the collection of $\V$-pointlike subsets of a semigroup $S$ forms a downward closed subsemigroup $\PL_{\V}(S)$ of the power semigroup $2^S$. 
Moreover, the assignment $S \mapsto \PL_{\V}(S)$ has the following `monadic' property first observed by Henckell and Rhodes.
\begin{proposition}\label{prop:pointlike-monad}
Let $S$ be a finite semigroup. For any $\V$-pointlike subset $\mathcal{X}$ of the semigroup $\PL_{\V}(S)$, the union $\bigcup \mathcal{X}$ is a $\V$-pointlike subset of $S$.
\end{proposition}
\begin{proof}
Suppose that $\mathcal{X}$ is $\V$-pointlike in $\PL_{\V}(S)$, and let $\phi \colon S \relto V$ be a relational morphism with $V \in \V$. We show that there exists $v \in V$ such that $\bigcup \mathcal{X} \subseteq v\phi^{-1}$. Define the relation $\Phi \colon \PL_{\V}(S) \relto V$ by setting $X\Phi := \{v \in V \mid X \subseteq v\phi^{-1}\}$ for every $X \in \PL_\V(S)$. Then $\Phi$ is a relational morphism. Since $\mathcal{X}$ is $\V$-pointlike, pick $v \in V$ such that $\mathcal{X} \subseteq v\Phi^{-1}$. By definition of $\Phi$, this means that $X \subseteq v\phi^{-1}$ for every $X \in \mathcal{X}$. Hence, $\bigcup \mathcal{X} \subseteq v\phi^{-1}$.
\end{proof}

\subsection*{$\bH$-kernels} Recall that, for any finite group $G$, the \emph{$\bH$-kernel} $K_{\bH}(G)$ of $G$ is defined to be the smallest normal subgroup $N$ of $G$ such that $G/N$ belongs to $\bH$.  Observe that a group $G$ belongs to $\bH$ if and only if $K_{\bH}(G)$ is trivial. Therefore, membership in $\bH$ is decidable if and only if the $\bH$-kernel of any finite group is computable.

We will need the following two lemmas about groups and their $\bH$-kernels.
\begin{lemma}\label{lem:lift-group}
Let $S$ be a finite semigroup and $\phi \colon S \onto G$ a surjective semigroup homomorphism with $G$ a group. There exists a subgroup $H \leq S$ such that $H\phi = G$.
\end{lemma}
\begin{proof}
See, e.g., \cite[Prop.~4.1.44]{RS2009}.
\end{proof}

\begin{lemma}\label{lem:preserve-H-kernels}
For any onto group homomorphism $\phi \colon G \onto H$, $K_\bH(G)\phi = K_\bH(H)$.
\end{lemma}
\begin{proof}
The composite map $G \onto H \onto H/{K_\bH(H)}$ has kernel $K_\bH(H)\phi^{-1}$, which thus contains $K_\bH(G)$ since $H/{K_\bH(H)}$ is in $\bH$. Hence, $K_\bH(G)\phi \subseteq K_\bH(H)$. For the other inclusion, note that $G/{K_\bH(G)}$, which is in $\bH$, maps onto $H/{K_\bH(G)\phi}$, so that $H/{K_\bH(G)\phi}$ is in $\bH$. Therefore, $K_\bH(H) \subseteq K_\bH(G)\phi$, as required.
\end{proof}

The following proposition establishes a first connection between $\bH$-kernels and $\bHbar$-pointlike sets.
\begin{proposition}\label{prop:H-kernel-pointlike}
Let $S$ be a finite semigroup and let $G \leq S$ be a subgroup. Then $K_{\bH}(G)$ is an $\bHbar$-pointlike subset of $S$.
\end{proposition}
\begin{proof}

Let $\phi \colon S \relto T$ be a relational morphism with $T \in \bHbar$. Applying Lemma~\ref{lem:lift-group}  to the restriction of the first projection, $p_S \colon \#\phi \cap (G \times T) \onto G$, pick a subgroup $H \leq \#\phi$ such that $Hp_S = G$. Since $Hp_T$ is a subgroup of $T$ and $T \in \bHbar$, we have $Hp_T \in \bH$. Therefore, $K_{\bH}(H) \subseteq up_T^{-1}$, where $u$ is the unit of the group $Hp_T$. By Lemma~\ref{lem:preserve-H-kernels}, $K_{\bH}(G) = K_{\bH}(H)p_S \subseteq up_T^{-1}p_S = u\phi^{-1}$.
\end{proof}

\subsection*{Automata and flows}
In this paper, by an \emph{automaton}\footnote{A more descriptive term would be `complete finite deterministic automaton (without a specified set of final states)', but there is no danger of confusion, since this is the only kind of automaton we consider in this paper.} we mean a quadruple $\A = (Q,A,\tau,i)$, where $Q$ and $A$ are finite sets, $i \in Q$, and, for each $a \in A$, $\tau_a$ is a function from $Q$ to $Q$. Thus, $\A$ defines a right action of $A^+$ on $Q$, and we usually omit explicit reference to $\tau$: if $w = a_1\dots a_n \in A^+$ then $qw := q\tau_{a_1}\cdots\tau_{a_n}$.  The mapping $q\mapsto qw$ is denoted $\tau_w$.   
The \emph{transition semigroup} $\T_\A$ of $\A$ is the image of the homomorphism $w \mapsto \tau_w$, or equivalently the semigroup of transformations on $Q$ generated by $\{\tau_a \mid a \in A\}$.
If $\V$ is a variety of finite semigroups, then by a \emph{$\V$-automaton} we mean an automaton $\A$ such that $\T_\A \in \V$.

Let $T$ be a semigroup generated by a subset $A \subseteq T$. For a word $w \in A^+$, we write $w_T$ for the element of $T$ represented by $w$. Let $\A = (Q,A,\tau,i)$ be an automaton. A \emph{$T$-flow} on $\A$ is a function $\Phi \colon Q \to 2^{T^I}$ such that $I \in i\Phi$, and for all $a \in A$, $q \in Q$, $(q\Phi) a \subseteq (qa)\Phi$. Note that, if $\Phi$ is a flow, then for any $w \in A^+$, $q \in Q$, we have $(q\Phi)w_T \subseteq (qw)\Phi$, by a straightforward induction on the length of $w$.
Flows on $\V$-automata allow for the following useful alternative characterization of $\V$-pointlike sets\footnote{This is a variation on the results of~\cite{HRS2012}, where more complicated problems are considered.}.
\begin{proposition}\label{prop:pointlike-flow}
Let $\V$ be a variety of finite semigroups, let $T$ be a finite semigroup generated by $A \subseteq T$, and let $X \subseteq T$. The following are equivalent:
\begin{enumerate}
\item the set $X$ is $\V$-pointlike;
\item for any $T$-flow $\Phi$ on a $\V$-automaton $\A$,  $X \subseteq q\Phi$ for some $q \in Q$.
\end{enumerate}
\end{proposition}
\begin{proof}
We first prove (1) $\Rightarrow$ (2). Suppose that $X$ is $\V$-pointlike, and let $\Phi$ be a $T$-flow on a $\V$-automaton $\A$. Define the relation $\phi \colon T \relto \T_\A$ by $t\phi := \{\tau_w \mid w \in A^+, w_T = t\}$, for $t \in T$. Then $\phi$ is a relational morphism.  Since $X$ is $\V$-pointlike and $\A$ is a $\V$-automaton, pick $\tau_w \in \T_\A$ such that $X \subseteq (\tau_w)\phi^{-1}$. We claim that $X \subseteq (iw)\Phi$. Indeed, if $x \in X$, then since $\tau_w \in x\phi$, pick $v \in A^+$ such that $v_T = x$ and $\tau_v = \tau_w$. Since $\Phi$ is a flow, we have $x = Ix \in (i\Phi) v_T \subseteq (iv)\Phi = (iw)\Phi$.

For the converse, let $\psi\colon T\to V$ be a relational morphism  with $V\in \V$ and suppose that $X\neq \emptyset$ is as in (2).  Extend $\psi$ to a relational morphism $\psi\colon T^I\relto V^I$ by putting $I\psi=\{I\}$.  Fix $v_a\in a\psi$ for each $a\in A$ and define $\A=(V^I,A,\tau,I)$ where $v\tau_a = vv_a$ for $v\in V^I$ and $a\in A$.  Then the transition semigroup of $\A$ is the subsemigroup $\langle v_a \mid a\in A\rangle$ of $V$ and hence belongs to $\V$.  Define a flow $\Phi\colon V^I\to 2^{T^I}$ by $v\Phi = v\psi^{-1}$ for $v\in V^I$.  Note that $(v\Phi) a=v\psi^{-1}a\subseteq v\psi^{-1}v_a\psi^{-1}\subseteq (vv_a)\psi^{-1}=(v\tau_a)\Phi$ and $I\in I\Phi$.  Thus $\Phi$ is a flow and so $X\subseteq v\Phi$ for some $v\in V^I$.  As $\emptyset\neq X\subseteq T$, we conclude that $X\subseteq v\Phi=v\psi^{-1}$ with $v\in V$.  Thus $X$ is $\V$-pointlike.
\end{proof}

\subsection*{The Sch\"utzenberger group}\label{subsec:schutzenberger-group}
%
Let $S$ be a finite semigroup. For any subset $X$ of $S$, the \emph{right stabilizer} $\St_R(X)$  of $X$ is the set of elements $s \in S^I$ such that $Xs \subseteq X$; this is a submonoid of $S^I$. Let $L$ be an $\L$-class in $S$. The stability of finite semigroups implies that, for any $\sH$-class $H \subseteq L$, we have $\St_R(H) = \St_R(L)$. Thus, for any $\sH$-class $H\subseteq L$, the monoid $\St_R(L)=\St_R(H)$ acts on $H$ via right multiplication.
Moreover, for any $s \in \St_R(H)$, the function $r_s|_H$, right multiplication by $s$ restricted to $H$, is a permutation of $H$ which is either the identity or fixes no points of $H$ at all. Hence, the faithful quotient of the action of $\St_R(H)$ on $H$ is a group, which is denoted by $\Gamma_R(H)$ and is called the (right) \emph{Sch\"utzenberger group} of $H$. The action of $\Gamma_R(H)$ on $H$ is transitive and has \emph{trivial point stabilizers}, i.e., if $g \in \Gamma_R(H)$, and $xg = x$ for some $x \in H$, then $g$ is equal to the unit of $\Gamma_R(H)$. In particular, $|\Gamma_R(H)|=|H|$.  In the event that an $\sH$-class $H$ is itself a subgroup of $S$, it is isomorphic to $\Gamma_R(H)$. See, e.g., \cite[Sec.~A.3.1]{RS2009} for proofs of the above facts.

\begin{lemma}\label{lem:lift-Schutz-group}
Let $L$ be an ${\sL}$-class in a finite semigroup $S$.
\begin{enumerate}
\item For any ${\sH}$-classes $H, H' \subseteq L$, the kernels of the natural homomorphisms $\St_R(L) \onto \Gamma_R(H)$ and $\St_R(L) \onto \Gamma_R(H')$ are equal; in particular, $\Gamma_R(H) \cong \Gamma_R(H')$.
\item There exists a subgroup $G_L$ of $\St_R(L)$ such that, for every ${\sH}$-class $H \subseteq L$, the restriction of the natural homomorphism $\St_R(L) = \St_R(H) \onto \Gamma_R(H)$ to $G_L$ is surjective.
\end{enumerate}
\end{lemma}
\begin{proof}
(1) Let $s, s' \in \St_R(L)$, and suppose that $r_s|_{H} = r_{s'}|_{H}$. Pick $x \in H$, and let $y \in H'$ be arbitrary. Since $y \sL x$, pick $t \in S^I$ such that $y = tx$. Then $ys = txs = txs' = ys'$. Thus, $r_s|_{H'} = r_s'|_{H'}$.

(2) By (1), if a subset $X \subseteq \St_R(L)$ is mapped onto $\Gamma_R(H)$ under the natural homomorphism for some ${\sH}$-class $H \subseteq L$, then in fact $X$ is mapped onto $\Gamma_R(H')$ under the natural homomorphism for every ${\sH}$-class $H' \subseteq L$. By Lemma~\ref{lem:lift-group}, there exists such a subset which is moreover a subgroup.
\end{proof}

The left Sch\"utzenberger group $\Gamma_L(H)$ is defined analogously to $\Gamma_R(H)$. There exists a function $\alpha$ such that, for any $x \in H$, $g \in \Gamma_R(H)$, $xg = (g\alpha)x$, and $\alpha$ is an anti-isomorphism from $\Gamma_R(H)$ to $\Gamma_L(H)$, cf.~\cite[Lemma~A.3.12]{RS2009}.

We call an element $s \in S$ an \emph{$\bH$-element} if the group $\Gamma_R(H_s)$ is in the variety $\bH$. In light of Lemma~\ref{lem:lift-Schutz-group}(1), any $\sL$-class either consists entirely of $\bH$-elements, or does not contain any $\bH$-elements.

We record two further important properties of stabilizers and Sch\"utzenberger groups that we will use in what follows.
\begin{lemma}\label{lem:one-for-all-right}
If $t \in S^I$, $x \in S$, and $xt \in H_x$, then $t \in \St_R(H_x)$.
\end{lemma}
\begin{proof}
Let $y \sL x$. Then  $yt \sL xt\sL x$, and so $t\in \St_R(L_x)=\St_R(H_x)$.
\end{proof}

The following is a reformulation of $\Gamma_R(H)$ having trivial point stabilizers.

\begin{lemma}\label{lem:trivial-point-stabilizers}
If $g, g' \in \Gamma_R(H)$ and, for some $x \in H$, $xg = xg'$, then $g = g'$.
\end{lemma}

\section{Computing $\bHbar$-pointlikes}

\subsection*{Statement of the main result}
The main theorem of this paper, Theorem~\ref{thm:main}, gives a characterization of the $\bHbar$-pointlike subsets of a finite semigroup $T$, which is effective whenever the membership problem of $\bH$ is decidable. The following definition is crucial; it generalizes the definition of `$C^\omega$' for aperiodic pointlike sets \cite[Def.~3.4]{Hen1988} (also see `$\mathrm{Sat}$' in \cite[Sec.~4.2]{PlaZei2016FO} and \cite[Sec.~4]{GooSte2017merge}) and of `$\mathrm{CP}_\pi(T)$' for $\mathbf{G}_\pi$-pointlike sets \cite[Thm.~2.3]{HRS2010AP}. %
\begin{definition}\label{def:saturated}
Let $T$ be a finite semigroup. A subset $S$ of $2^T$ is \emph{$\bHbar$-saturated} if
\begin{enumerate}[label={(\roman*)}]
	\item $S$ is a subsemigroup,
	\item $S$ is closed downward in the inclusion order, and
	\item $S$ is closed under taking unions of $\bH$-kernels of subgroups, i.e., for any subgroup $G \subseteq S$, $\bigcup K_\bH(G) \in S$.
\end{enumerate}
Since an intersection of $\bHbar$-saturated subsets of $2^T$ is again $\bHbar$-saturated, any subset $U$ of $2^S$ is contained in a smallest $\bHbar$-saturated set, which we call its \emph{$\bHbar$-saturation}, and denote by $\Sat_{\bHbar}(U)$.
\end{definition}
Write $\eta \colon T \to 2^T$ for the injective homomorphism defined by $t\eta := \{t\}$.
Our main theorem, to be proved in Section~\ref{sec:mainproof}, is the following characterization of the collection of $\bHbar$-pointlike sets of a finite semigroup $T$.
\begin{theorem}\label{thm:main}
For any finite semigroup $T$, $\PL_{\bHbar}(T) = \Sat_{\bHbar}(T\eta)$. That is, the set of $\bHbar$-pointlike subsets of $T$ is the $\bHbar$-saturation of the set of singletons.
\end{theorem}


For the rest of the paper, we will fix the finite semigroup $T$, and we will also fix $S := \Sat_{\bHbar}(T\eta)$. 
The difficult direction of the proof is to show that any $\bHbar$-pointlike subset of $T$ lies in $S$. The proof of this direction will consist of two main parts: in Section~\ref{sec:blowup-flow}, we construct a $T$-flow $\Phi$ on an automaton $\A$ with alphabet $T$. Then, in Section~\ref{sec:in-Hbar}, we prove that $\A$ is an $\bHbar$-automaton. Finally, in Section~\ref{sec:mainproof} we prove Theorem~\ref{thm:main}.

\subsection*{Applications}
We state some applications of Theorem~\ref{thm:main} here.
If $\bH$ has a decidable membership problem, then one can compute the $\bH$-kernel of any finite group and hence one can effectively test if a subsemigroup of $2^T$ is $\bHbar$-saturated.  Thus one can effectively find $\Sat_{\bHbar}(T\eta)$ under this hypothesis.  Therefore, Theorem~\ref{thm:main} has the following corollary.

\begin{corollary}\label{cor:is-dec}
Let $\bH$ be a variety of finite groups.  Then the following are equivalent.
\begin{enumerate}
  \item $\bH$ has decidable membership.
  \item $\bHbar$ has decidable membership.
  \item $\bHbar$-pointlikes are computable.
\end{enumerate}
\end{corollary}

This implies Henckell's result on the decidability of aperiodic pointlikes~\cite{Hen1988} by letting $\bH$ be the variety containing only the trivial group.  If $\pi$ is a set of primes, then $\mathbf G_{\pi}$ denotes the variety of finite groups whose order is divisible only by primes in $\pi$.  If $\pi=\emptyset$, then $\mathbf G_{\pi}$ is the trivial variety.  Thus varieties of the form $\overline{\mathbf G_{\pi}}$ include the variety of aperiodic semigroups.  Clearly, $\mathbf G_{\pi}$ has decidable membership if and only if $\pi$ is recursive.  Thus Corollary~\ref{cor:is-dec} implies the main result of~\cite{HRS2010} that $\overline{\mathbf G_{\pi}}$ has decidable pointlikes if $\pi$ is recursive. Let $\mathbf G_{\mathrm{sol}}$ denote the variety of finite solvable groups.

\begin{corollary}\label{c:Gsol}
The variety $\overline{\mathbf G_{\mathrm{sol}}}$ has decidable pointlikes.
\end{corollary}

The second author showed in~\cite{Slice} that if $\mathbf V$ has decidable pointlikes and $\mathbf W$ is locally finite with computable free objects (e.g., if $\mathbf W$ is generated by a finite semigroup), then the join $\mathbf V\vee \mathbf W$ has decidable membership; see also~\cite{Almeida99}.

\begin{corollary}\label{c:join}
If $\mathbf H$ is a decidable variety of finite groups and $\mathbf W$ is a locally finite variety of finite semigroups with computable free objects, then $\bHbar\vee \mathbf W$ is decidable.
\end{corollary}

Recall that if $\mathbf V$ and $\mathbf W$ are varieties of finite semigroups, then their \emph{Mal'cev product} $\mathbf V\mbox{\textcircled{\footnotesize{$m$}}} \mathbf W$ is the variety consisting of all finite semigroups $S$ admitting a relational morphism $\varphi\colon S\relto T$ with $T\in \mathbf W$ and $e\varphi^{-1}\in \mathbf V$ for all idempotents $e\in T$.   Let $\mathbf A$ denote the variety of aperiodic semigroups. Henckell proved in~\cite{Henckell04} (see also~\cite{HRS2010} for a simpler proof) that if $\mathbf A\mbox{\textcircled{\footnotesize{$m$}}}\mathbf W=\mathbf W$ (or, equivalently, the $\mathbf W$-recognizable languages are closed under concatenation~\cite{Str79}) and $\mathbf W$ has decidable pointlikes, then $\mathbf V\mbox{\textcircled{\footnotesize{$m$}}}\mathbf W$ has decidable membership for any variety of finite semigroups $\mathbf V$ with decidable membership problem.  Since $\mathbf A\mbox{\textcircled{\footnotesize{$m$}}}\bHbar=\bHbar$, we obtain the following.

\begin{corollary}\label{c:malcev}
Let $\mathbf H$ be a decidable variety of finite groups and $\mathbf V$ a decidable variety of finite semigroups.  Then $\mathbf V\mbox{\textcircled{\footnotesize{$m$}}}\bHbar$ has decidable membership.
\end{corollary}

Finally, we mention the separation problem.  If $\mathbf V$ is a variety of finite semigroups, then the \emph{separation problem} for $\mathbf V$ asks for an algorithm which, given two disjoint regular languages $L_1,L_2$ over an alphabet $A$, decides whether there is a $\mathbf V$-recognizable language $L\subseteq A^+$ with $L_1\subseteq L$ and $L\cap L_2=\emptyset$ (i.e., $L_2\subseteq A^+\setminus L$); we call $L$ a \emph{$\mathbf V$-separator}.  Clearly, there is a $\mathbf V$-separator for a regular language and its complement if and only if the language is $\mathbf V$-recognizable and hence decidability of the separation problem for $\mathbf V$ implies decidability of the membership problem for $\mathbf V$.  It follows from a result of Almeida~\cite{Almeida99} that the separation problem for $\mathbf V$ is equivalent to the decidability of $\mathbf V$-pointlike pairs.

\begin{corollary}\label{c:sep}
If $\mathbf H$ is a decidable variety of finite groups, then the separation problem for $\bHbar$ is decidable.
\end{corollary}

The separation problem for $\bHbar$ often has a logical interpretation.  The regular languages are precisely the languages definable in second order monadic logic by B\"uchi's theorem~\cite{Str1994}.  The Sch\"utzenberger-McNaughton-Pappert theorem shows that the first order definable languages are precisely the $\mathbf A$-recognizable languages (see~\cite{Str1994}).  Thus Henckell's theorem on the decidability of aperiodic pointlikes implies that it is decidable if the languages defined by two second order monadic formulas can be separated by a first order definable language.  Corollary~\ref{c:Gsol} admits a similar logical reinterpretation.  If $\phi$ is a first order formula with free variable $x$, then the modular quantifier $\exists^{(r,n)}x\phi(x)$ is defined to be true for a word $w$ if the number of positions $j$ in $w$ for which $\phi(j)$ is true is congruent to $r$ modulo $n$.  More generally, if $\phi$ is a formula using first order and modular quantifiers, with a single free variable $x$, then  $\exists^{(r,n)}x\phi(x)$ is defined in a similar fashion.

Straubing proved that the languages definable in first order logic with modular quantifiers are precisely the $\overline{\mathbf G_{\mathrm{sol}}}$-languages; see~\cite{Str1994}.

\begin{corollary}\label{c:mod.quant}
The separation problem is decidable for first order logic with modular quantifiers.
\end{corollary}

If $\pi$ is a set of primes, then the $\overline{\mathrm G_{\pi}\cap \mathrm{G_{\mathrm{sol}}}}$ has an analogous  logical interpretation, but one is restricted to modular quantifiers whose moduli are divisible only by primes in $\pi$, cf.~\cite{Str1994}.  Thus, for any recursive set of primes $\pi$, we can decide the separation problem for the fragment of first order logic with modular quantifiers where the moduli are divisible only by primes in $\pi$.

\section{Construction of automaton and flow}\label{sec:blowup-flow}
The aim of this section is to use the semigroup $S = \Sat_{\bHbar}(T\eta)$ to construct an automaton $\A$, and a $T$-flow $\Phi$ on $\A$. We will then prove in the next section that the automaton $\A$ that we construct here is in fact an $\bHbar$-automaton.

\subsection*{$\L$-chains, strict $\L$-chains, $\bH$-elements}
Throughout this section and the next, we will work with the free monoid $S^*$ over $S$, and we identify $T$ with its image in $S$ under $\eta$. To avoid confusion with the multiplication operations in $S$ and $T$, we write words in $S^*$ as row vectors $\q = (q_n,\dots,q_1)$, where $q_i \in S$ for each $1 \leq i \leq n$. If $s \in S$, we write $s$ instead of $(s)$ for a one-element vector. We denote the concatenation of two words $\q, \q' \in S^*$ by $\q \cdot \q'$. Words in $S^*$ are read \emph{from right to left}, so that $q_1$ is the \emph{first} and $q_n$ is the \emph{last} letter of $\q = (q_n,\dots,q_1)$.
We define the operation `last letter', $\omega \colon S^+ \to S$, by $(q_n,\dots,q_1)\omega := q_n$.

An \emph{$\L$-chain} (or \emph{flag}) in $S$ is a (possibly empty) finite word $\vec{q} = (q_n,\dots,q_1) \in S^*$ such that $q_{i+1} \leqL q_i$ for every $1 \leq i < n$. The set of $\L$-chains in $S$ will be denoted by $\sFbar$.
An $\L$-chain $\vec{q} \in \sFbar$ is \emph{strict} if, for every $1 \leq i < n$, $q_{i+1} <_{\L} q_i$, i.e., $q_{i+1}$ is not $\L$-equivalent to $q_i$. The set of strict $\L$-chains in $S$ is denoted by $\sF$.
Note that $\sF$ is a \emph{finite} set.\footnote{In \cite{HRS2010AP}, these sets were denoted $\sF(S)$ and $\sFbar(S)$, but we can safely omit this reference to $S$, since $S$ is fixed throughout the paper.}

Recall that $s \in S$ is an $\bH$-element if the Schützenberger group  $\Gamma_L(H_s)$ of the $\sH$-class of $s$ lies in $\bH$. We denote by $\sFbar_\bH$ the set of $\L$-chains that consist entirely of $\bH$-elements, and by $\sF_\bH$ the set of strict $\L$-chains that consist entirely of $\bH$-elements.

There is a natural \emph{retraction} $\rho \colon \sFbar \to \sF$, which makes an $\L$-chain strict by erasing all but the last element of a sequence of $\L$-equivalent elements. That is, if $\q = (q_n,\dots,q_{i+1},q_i,\dots,q_1)$ is an $\L$-chain with $q_i  \sL  q_{i+1}$ for some $1 \leq i < n$, then we say that $\q$ \emph{reduces in one step} to the $\L$-chain $\q' = (q_n,\dots,q_{i+1},q_{i-1},\dots,q_1)$. If $\q'$ can be obtained from $\q$ by applying a sequence of one-step reductions, we say that $\q$ \emph{reduces to} $\q'$. The one-step reduction relation forms a confluent rewriting system: if $\q$ reduces to $\q_1$ and $\q_2$, then there is $\q'$ to which both $\q_1$ and $\q_2$ reduce. Thus, any $\q \in \sFbar$ reduces to a unique strict $\L$-chain, $\q\rho$. 
The following lemma collects a few obvious properties of $\rho$ for future reference.

\begin{lemma}\label{lem:rho-properties}
For any $\q,\q' \in \sFbar$,
\begin{enumerate}
\item if $\q$ reduces to $\q'$, then $\q\rho = \q'\rho$;
\item $\vec{q}\rho = \vec{q}$ if and only if $\vec{q}$ is strict;
\item $\q\rho\omega = \q\omega$;
\item $(\q \cdot \q')\rho = (\q \cdot \q'\rho)\rho = (\q\rho \cdot \q')\rho$.
\item if $\q \in \sFbar_\bH$, then $\q\rho \in \sF_\bH$.
\end{enumerate}
\end{lemma}

We will define an automaton $\A = (\sF_{\bH},T,\tau,\eps)$, where the \emph{states} of $\A$ are strict $\sL$-chains of $\bH$-elements in $S$; the \emph{alphabet} of $\A$ is $T$; the \emph{initial state} $i$ of $\A$ is the empty $\sL$-chain, $\eps$; the definition of the \emph{transition functions} $\tau_t$ for $t \in T$ will be given in Definition~\ref{def:tau} below.

\subsection*{Blowup operators}
Our first step towards defining the transition functions $\tau_t$ will be to construct a \emph{blowup operator} $b \colon S \to S$.
\begin{definition}\label{def:blowup}
A function $b \colon S \to S$ is a \emph{pre-blowup operator} if:
\begin{enumerate}[label={(\roman*)}]
\item \label{itm:multiplier} for any $\L$-class $L$, there exists $m_L \in S^I$ such that $sb = sm_L$ for all $s \in L$;
\item \label{itm:H-elements} for any $s \in S$, if $s$ is an $\bH$-element, then $sb = s$, and if $s$ is not an $\bH$-element, then $sb <_{\sH} s$;
\item \label{itm:blows-up} for any $s \in S$, $s \subseteq sb$.
\end{enumerate}
If $b$ is a pre-blowup operator, then a function $m \colon S/{\L} \to S^I$, $L \mapsto m_L$, such that $sb = sm_{L_s}$  for all $s \in S$, 
is called  a \emph{multiplier for $b$}.
A pre-blowup operator $b$ is a \emph{blowup operator} if moreover
\begin{enumerate}
\item[(iv)] \label{itm:idempotent} $b$ is  idempotent, i.e., $sbb = sb$ for every $s \in S$.
\end{enumerate}
\end{definition}
\begin{proposition}\label{prop:preblowup-exists}
The semigroup $S$ admits a pre-blowup operator.
\end{proposition}
\begin{proof}
For any ${\L}$-class $L$ in $S$, pick a subgroup $G_L \leq \St_R(L)$ as in Lemma~\ref{lem:lift-Schutz-group}(2). If $L$ consists of $\bH$-elements, define $m_L := I$. Otherwise, $L$ does not contain any $\bH$-elements, so the group $G_L$ must be non-trivial, and therefore $G_L$ does not contain $I$. Define $m_L := \bigcup K_{\bH}(G_L)$. Note that $m_L$ is in $S$ since $S$ is $\bHbar$-saturated. Now, for every $s \in S$, define $sb_0 := sm_{L_s}$. Note that, for every $s \in S$,
\[sb_0 = \bigcup \{sk \mid k \in K_{\bH}(G_L)\} = \bigcup \{sk \mid k \in K_{\bH}(\Gamma_R(H_s))\},\]
where the last equality follows from Lemma~\ref{lem:preserve-H-kernels}. We will now show that $b_0$ is a pre-blowup operator, by verifying (i) -- (iii) in Def.~\ref{def:blowup}.

(i) By definition of $b_0$.

(ii) If $s$ is an $\bH$-element, then $sb_0 = sI = s$. Now suppose that $s$ is not an $\bH$-element. We will show that (a) $sb_0 \leq_{\sH} s$ and (b) $sb_0 \not\in H_s$. %
(a) Clearly, $sb_0 = sm_{L_s} \leq_{\sR} s$. To prove that $sb_0 \leq_{\sL} s$, by the dual of Lemma~\ref{lem:lift-Schutz-group}(2), pick a subgroup $G' \leq \St_L(R_s)$ such that the restriction of $\St_L(R_s) \onto \Gamma_L(H)$ is surjective for every ${\sH}$-class $H \subseteq R_s$. Let $m' := \bigcup K_{\bH}(G')$. We show that $m's = sm_{L_s}$.
Let $\bar{k} \in K_{\bH}(G_{L_s})$ be arbitrary, and let $k$ be its image under the natural homomorphism to $\Gamma_R(H_s)$. Then $k$ is in $K_{\bH}(\Gamma_R(H_s))$ by Lemma~\ref{lem:preserve-H-kernels}, and therefore $k\alpha$ is in $K_{\bH}(\Gamma_L(H_s))$, where $\alpha$ is the anti-isomorphism between $\Gamma_R(H_s)$ and $\Gamma_L(H_s)$. By Lemma~\ref{lem:preserve-H-kernels}, there exists $\bar{k}' \in K_{\bH}(G')$ which is mapped onto $k\alpha$ by the natural homomorphism from $G'$ to $\Gamma_L(H_s)$. Now $s\bar{k} = sk = (k\alpha) s = \bar{k}'s$, so that $s\bar{k}$ is contained in $m's$. Since $\bar{k}$ was arbitrary, we have proved that $sm_{L_s} \subseteq m's$. The proof that $m's \subseteq sm_{L_s}$ is symmetric. Thus, $sb_0 = m's \leq_{\sL} s$. %
(b) Towards a contradiction, suppose that $sb_0 \in H_s$. For any $k_0 \in K_\bH(\Gamma_R(H_s))$,
\[ (sb_0)k_0 = \bigcup \{skk_0 \ : \ k \in K_\bH(\Gamma_R(H_s))\} = 
sb_0.\]
Therefore, since the point stabilizers of the action of $\Gamma_R(H_s)$ on $H_s$ are trivial, $K_\bH(\Gamma_R(H_s))$ is a trivial group. But this means that $s$ is an $\bH$-element, which is a contradiction.

(iii) If $s$ is a $\bH$-element, we have equality. If not, notice that $K_\bH(G_{L_s})$ contains the identity $u$ of $G_L$ and so $s = su \subseteq sm_{L_s} = sb_0$.
\end{proof}
\begin{remark}
Other pre-blowup operators exist.  For instance, if $s$ is not an $\bH$-element and $g\in K_{\bH}(G_{L_S})$ represents a non-trivial element of $K_\bH(\Gamma_R(H_s))$, one could use $m_{L_s}=\bigcup \langle g\rangle$ instead (and if $s$ is an $\bH$-element, one still uses $m_{L_s}=I$). This approach was taken in~\cite{HRS2010AP}.
\end{remark}

The following lemma collects a few useful properties of pre-blowup operators.
\begin{lemma}\label{lem:preblowup-L}
Let $b, b' \colon S \to S$ be pre-blowup operators and let $m, m' \colon S/{\L} \to S$ be multipliers for $b$ and $b'$, respectively. Then, for any $q, s, s' \in S$:
\begin{enumerate}
\item if $s  \sL  s'$, then $sb  \sL  s'b$;
\item the composition $bb'$ is a pre-blowup operator;
\item if $q \leqL s$, then $q \subseteq qm_{L_s}$;
\item if $q \leqL s$ and $s$ is an $\bH$-element, then $q = qm_{L_s}$;
\item if $qb \sR q$, then $qb = q$;
\item if $b$ is idempotent (i.e., a blowup operator), then the image of $b$ is the set of $\bH$-elements of $S$.
\end{enumerate}
\end{lemma}
\begin{proof}
(1) Since $\L$ is a right congruence and $L_s = L_{s'}$, $sb = sm_{L_s}  \sL  s'm_{L_{s}} = s'b$.
(2) We first define a multiplier $n$ for $bb'$. By (1), for any $\L$-class $L$, the image $Lb$ is contained in a unique $\L$-class, $L'$; define $n_L := m_L m'_{L'}$. Then, for any $s \in L$, $sbb' = sm_Lb' = sm_Lm'_{L'} = sn_L$, so $n$ is a multiplier for $bb'$. Properties (ii) and (iii) in Def.~\ref{def:blowup} clearly hold for $bb'$.
(3) \& (4) Since $q \leqL s$, pick $\alpha \in S$ such that $q = \alpha s$. Since $s \subseteq sb = sm_{L_s}$, we obtain $q = \alpha s \subseteq \alpha s m_{L_s} = qm_{L_s}$. If $s$ is an $\bH$-element, then $sb = s$, and therefore $qm_{L_s} = \alpha sm_{L_s} = \alpha sb = \alpha s = q$. (5) Since $qb \leqH q$, if $qb \sR q$, then $qb  \J  q$, so stability yields $qb  \sL  q$. Therefore, $qb  \sH  q$, so by Def.~\ref{def:blowup}\ref{itm:H-elements}, $q$ is an $\bH$-element, and hence $qb = q$. (6) Note that Def.~\ref{def:blowup}\ref{itm:H-elements} implies that the set of fixed points of $b$ is the set of $\bH$-elements of $S$. The image of an idempotent operator is its set of fixed points. 
\end{proof}
\begin{proposition}\label{prop:blowup-exists}
The semigroup $S$ admits a blowup operator.
\end{proposition}
\begin{proof}
By Proposition~\ref{prop:preblowup-exists}, pick a pre-blowup operator $b_0$ on $S$. Since pre-blowup operators are closed under composition by Lemma~\ref{lem:preblowup-L}(2), the idempotent power under composition, $b := b_0^\omega$, is a blowup operator.
\end{proof}
From here on out, we fix a blowup operator $b$ on $S = \Sat_{\bHbar}(T\eta)$. Our next step towards defining the transition functions of our automaton is to extend the blowup operator $b$ to an operator, $B$, that will act on $S^*$.
\begin{definition}\label{def:diagonal-and-Bhat}
For any $s \in S^I$, define the \emph{diagonal operator associated to $s$}, $\Delta_s \colon S^* \to S^*$, by $\q\Delta_s := (q_ns,\dots,q_1s)$, for any $\q = (q_n,\dots,q_1) \in S^*$. In particular, $\eps \Delta_s := \eps$.

Fix a multiplier $m \colon S/{\L} \to S^I$ for $b$. Recursively define the length-preserving function $B \colon S^* \to S^*$ by $\eps B := \eps$, and, for any $\q \in S^*$ and $s \in S$,
\[(\q \cdot s) B := (\q\Delta_{m_{L_s}})B \cdot sb.\]
\end{definition}
\begin{example}\label{exa:Bhat-up-to-three}
To illustrate the definition of $B$, we compute $B$ for words of length up to $2$. By definition, $\eps B = \eps$. For any word of length one, $q \in S$, we compute
\[ qB = \eps\Delta_{m_{L_q}}B \cdot qb = \eps B \cdot qb = \eps \cdot qb = qb = qm_{L_q}.\]
Now, for a word $(q_2,q_1)$ of length $2$,
\[ (q_2,q_1)B = (q_2m_{L_{q_1}}b, q_1b) = (q_2m_1m_2,q_1m_1),\]
where $m_1 := m_{L_{q_1}}$ and $m_2 := m_{L_{q_2m_1}}$.
%
%
\end{example}
While the operators $\Delta_s$ and $B$ are defined on all of $S^*$, we will mostly be interested in their action on $\L$-chains.
How $B$ acts on $\L$-chains does not depend on the specific choice of a multiplier $m$. Indeed, suppose $m$ and $m'$ are multipliers for $b$ which are used to define length-preserving functions $B$ and $B'$, respectively. Then, for any $q, s \in S$ with $q \leq_{\L} s$, pick $x \in S^I$ such that $q = xs$, and note $qm_{L_s} = xsm_{L_s} = xsb = xsm'_{L_s} = qm'_{L_s}$. Inductively, we conclude that $\q B = \q B'$ for any $\L$-chain $\q$.

\begin{proposition}\label{prop:Delta-Bhat-preserve-flags}
For any ${\L}$-chain $\q$,
\begin{enumerate}
\item the word $\q\Delta_s$ is an ${\L}$-chain, for any $s \in S^I$;
\item the word $\q B$ is an ${\L}$-chain of $\bH$-elements;
\item if $\q$ consists entirely of $\bH$-elements, then $\q B = \q$;
\item the $i^\th$ coordinate of $\q$ is contained in the $i^\th$ coordinate of $\q B$.
\end{enumerate}
In particular, the restriction $B|_{\sFbar}$ of $B$ to ${\L}$-chains is idempotent.
\end{proposition}
\begin{proof} Let $\q$ be an ${\L}$-chain.
(1) is clear from the fact that $q \leqL q'$ implies $qs \leqL q's$.
We show (2)--(4) by induction on the length of $\q$. The case where $\q = \eps$ is trivial. The case where $\q$ is a one-letter word is clear from Def.~\ref{def:blowup} and Lemma~\ref{lem:preblowup-L}(6).
Now let $n \geq 1$, assume that (2)--(4) have been proved for words of length at most $n$, and let $\q' = \q \cdot s$ be an $\L$-chain, where $\q = (q_n,\dots,q_1)$. By (1), $\q \Delta_{m_{L_s}}$ is an $\L$-chain, so $\q \Delta_{m_{L_s}}B$ is an $\L$-chain of $\bH$-elements by induction. Moreover, the first element of $\q \Delta_{m_{L_s}}B$ is $q_1m_{L_s}b$, and
\[q_1m_{L_s}b \leqL q_1m_{L_s} \leqL sm_{L_s} = sb,\]
where the first inequality holds by Def.~\ref{def:blowup}(ii), and the second inequality holds because $q_1 \leqL s$, since $\q \cdot s$ is an $\L$-chain. Also, $sb$ is an $\bH$-element by Lemma~\ref{lem:preblowup-L}(6), so $\q' B = \q \Delta_{m_{L_s}}B \cdot sb$ is an $\L$-chain of $\bH$-elements, establishing (2). For (3), assume $\q'$ consists of $\bH$-elements. Then Lemma~\ref{lem:preblowup-L}(4) implies that $\q \Delta_{m_{L_s}} = \q$, and by the induction hypothesis, $\q B = \q$. Therefore, since also $sb = s$ by Def.~\ref{def:blowup}(ii), we have $\q' B = \q'$. Finally, for (4), note that the first coordinate $s$ of $\q'$, is contained in the first coordinate $sb$ of $\q' B$, by Def.~\ref{def:blowup}(iii). For any $1 < i \leq n + 1$, the $i^\mathrm{th}$ coordinate of $\q'$ is $q_{i-1}$, which is $\L$-below $s$. By Lemma~\ref{lem:preblowup-L}(3), $q_{i-1}$ is contained in $q_{i-1}m_{L_s}$. By the induction hypothesis applied to the $\L$-chain $\q\Delta_{m_{L_s}}$, we have that $q_{i-1}m_{L_s}$ is contained in the $(i-1)^\mathrm{th}$ coordinate of $\q\Delta_{m_{L_s}}B$, which is, by definition, the $i^{\mathrm{th}}$ coordinate of $\q'B$.
\end{proof}

As a last step before we define the automaton, we now show that the diagonal operators and $B$ interact well with the retraction $\rho$ from $\L$-chains to strict $\L$-chains, in the following sense.
\begin{definition}
Let $F \colon \sFbar \to \sFbar$ be a function. We say that $F$ \emph{respects $\rho$} if, for any $\q \in \sFbar$, $\q \rho F \rho = \q F \rho$.
\end{definition}
Observe that if $F$ and $G$ respect $\rho$, then so does $FG$, since
\[\rho FG \rho = \rho F \rho G \rho = F \rho G \rho = FG \rho.\]

\begin{proposition}\label{prop:Delta-Bhat-rho}
The diagonal operators $\Delta_s$, for $s$ in $S^I$, and $B$ respect $\rho$.
\end{proposition}
\begin{proof}
First note that, for any $\L$-chain $\q$, the $\L$-chain $\q \Delta_s$ reduces to the $\L$-chain $\q\rho \Delta_s$. Hence, their images under $\rho$ are equal, and $\Delta_s$ respects $\rho$.
We now prove by induction on the length of $\q \in \sFbar$ that $\q \rho B \rho = \q B \rho$. If $\q$ has length $0$ or $1$, then there is nothing to prove. Let $n \geq 1$, assume by induction that the equality is proved for all words of length at most $n$, and let $\q' = \q\cdot s$, where $\q = (q_n,\dots,q_1)$, be an $\L$-chain. We write $m$ for $m_{L_s}$, and we distinguish two cases.

\emph{Case 1.} $q_1$ and $s$ are not ${\L}$-equivalent.

In this case, $\q'\rho = \q\rho \cdot s$. Therefore,
\begin{equation}\label{eq:rhoproof-1}
 \q'\rho B\rho = (\q\rho \cdot s)B\rho = (\q\rho\Delta_mB \cdot sb)\rho.
 \end{equation}
Note that
\begin{equation}\label{eq:rhoproof-2}
\q\rho\Delta_mB\rho = \q\rho\Delta_m\rho B\rho = \q\Delta_m\rho B\rho = \q \Delta_m B \rho,
\end{equation}
where we first apply the induction hypothesis to the word $\q\rho\Delta_m$, then the already established fact that $\Delta_m$ respects $\rho$, and then the induction hypothesis to the word $\q \Delta_m$. Hence,
\begin{align*}
\q'\rho B \rho &= (\q\rho\Delta_mB \cdot sb)\rho &\text{(using (\ref{eq:rhoproof-1}))}\\
&= (\q\rho\Delta_mB\rho \cdot sb)\rho &\text{(by Lemma~\ref{lem:rho-properties}(4))}\\
&= (\q \Delta_m B \rho \cdot sb)\rho &\text{(using (\ref{eq:rhoproof-2}))}\\
&= (\q \Delta_m B \cdot sb)\rho = \q'B\rho &\text{(by Lemma~\ref{lem:rho-properties}(4) and Def.~\ref{def:diagonal-and-Bhat}).}
\end{align*}

\emph{Case 2.} $q_1$ and $s$ are ${\L}$-equivalent.

In this case, $\q'\rho = \q\rho$. It then suffices to  show that $\q'B\rho = \q B\rho$ because we will have $\q'\rho B\rho = \q\rho B\rho =\q B\rho=\q'B\rho$ where the second equality uses the induction hypothesis.
Unravelling the definition of $B$ (Def.~\ref{def:diagonal-and-Bhat}) twice, we have
\begin{equation}\label{eq:rhoproof-3}
\q'B = \q^- \Delta_{m}\Delta_{m'}B \cdot q_1mb \cdot sb,
\end{equation}
where  $\q^- := (q_n,\dots,q_2)$ and $m' := m_{L_{q_1m}}$.

Now, since $m_{L_{q_1}} = m_{L_s} = m$, we have $q_1b = q_1m$. Since $b$ is idempotent, we obtain $q_1mb = q_1b^2 = q_1b = q_1m$. Therefore, $q_1b = q_1m$ is an $\bH$-element. Also, since ${\L}$ is a right congruence,
\begin{equation}\label{eq:rhoproof-4}
q_1mb = q_1m \L sm = sb.
\end{equation} 
Moreover, since $\q$ is an $\L$-chain, for every $2 \leq i \leq n$, we have $q_im \leqL q_1m$, so, by Lemma~\ref{lem:preblowup-L}(4), we have $q_imm' = q_im$. Thus,
\begin{equation}\label{eq:rhoproof-5}
\q^-\Delta_m\Delta_m' = \q^-\Delta_m.
\end{equation}
We now obtain, recalling that $m=m_{L_{q_1}}$,
\begin{align*}
\q'B\rho &= (\q^- \Delta_m \Delta_{m'} B \cdot q_1b)\rho &\text{(by (\ref{eq:rhoproof-3}) and (\ref{eq:rhoproof-4}))}\\
&= (\q^-\Delta_mB \cdot q_1b)\rho &\text{(by \ref{eq:rhoproof-5})} \\
&= \q B\rho. &{}  \qedhere
\end{align*}
\end{proof}

\subsection*{Definition of automaton and flow}
The transition functions of our automaton will act on \emph{strict} $\L$-chains. We first define functions, $\taubar_t$, which will act on $\L$-chains, and then compose $\taubar_t$ with $\rho$ to obtain at last the transition functions of our automaton.  Recall that we identify $T$ as a subsemigroup of $2^T$ via $t\mapsto t\eta=\{t\}$.

\begin{definition}\label{def:taubar}
For any $t \in T$ and $\q \in S^*$, define
\[ \q \taubar_t := (\q \Delta_{t} \cdot t)B.\]
\end{definition}

We show that $\taubar_{t}$ indeed acts on $\L$-chains.

\begin{lemma}\label{lem:taubar-on-states}
For any $t \in T$, $\taubar_t$ maps $\L$-chains to $\L$-chains of $\bH$-elements, and $\taubar_t$ respects $\rho$.
\end{lemma}
\begin{proof}
By Prop.~\ref{prop:Delta-Bhat-preserve-flags}(1), $\q \Delta_{t}$ is an $\L$-chain, and the first element of $\q \Delta_{t}$ is clearly $\L$-below $t$. Thus, $\q \Delta_{t} \cdot t$ is an $\L$-chain and so $\q \taubar_t$ is an $\L$-chain of $\bH$-elements by Prop.~\ref{prop:Delta-Bhat-preserve-flags}(2).
To see that $\taubar_t$ respects $\rho$, we note first that the function $\q \mapsto \q \cdot t$ respects $\rho$ by Lemma~\ref{lem:rho-properties}.
Thus $\taubar_t$, being a composition of 
 three functions  that respect $\rho$ (as $\Delta_t$ and $B$ respect $\rho$ by Proposition~\ref{prop:Delta-Bhat-rho}), respects $\rho$.
\end{proof}

We are now ready to define our automaton $\A$.
\begin{definition}[The automaton $\A$ and flow $\Phi$]\label{def:tau}
We define the automaton $\A$ with set of states $\sF_\bH$, alphabet $T$, initial state $\eps$, and, for every $t \in T$, transition function $\tau_t \colon \sF_\bH \to \sF_\bH$, defined by $\tau_t := \taubar_t|_{\sFbar} \rho.$
Note that $\tau_t$ is well-defined by Lemma~\ref{lem:rho-properties}(5) and Lemma~\ref{lem:taubar-on-states}.
Define $\Phi \colon \sF_\bH \to 2^{T^I}$ by $\q\Phi := \q\omega$ for $\q\neq \eps$ and $\eps\Phi=\{I\}$.
\end{definition}

\begin{proposition}\label{prop:Phi-is-flow}
The map $\Phi$ is a $T$-flow on the automaton $\A$.
\end{proposition}
\begin{proof}
Clearly, $I\in \eps\Phi$.  Also, for $t\in T$, we have that $\eps\Phi t=\{t\}\subseteq \{t\}b=(\eps\tau_t)\omega$.   
Let $\q \in \sF_\bH\setminus \{\eps\}$ and $t \in T$. We need to show that $(\q \omega) t \subseteq (\q \tau_t)\omega$. Note that $(\q \omega)t = (\q\Delta_{t} \cdot t)\omega$, and, using Lemma~\ref{lem:rho-properties}(3), $(\q \tau_t)\omega = (\q \Delta_{t} \cdot t)B\omega$. Thus, by Prop.~\ref{prop:Delta-Bhat-preserve-flags}(4), $(\q \omega) t \subseteq (\q \tau_t)\omega$.
\end{proof}

\section{$\A$ is an $\bHbar$-automaton}\label{sec:in-Hbar}
Let $\A$ be the automaton defined in Def.~\ref{def:tau}. The aim of this section is to prove (Corollary~\ref{cor:TA-in-Hbar}) that the transition semigroup $\T_\A$ of $\A$ is in $\bHbar$. To this end, we will define another semigroup $S^\bH$, which we will prove contains $\T_\A$ as a subsemigroup (Proposition~\ref{prop:SH-contains-taus}), and is itself in $\bHbar$ (Theorem~\ref{thm:SH-in-Hbar}).

\subsection*{Infinite wreath products and self-similarity}
Below, we recall a definition of the \emph{infinite wreath product} $R^\infty$ of a right transformation monoid $(X,R)$ and its action on the set $X^*$ of finite words over $X$. See \cite{GriNekSus2000} for more details, and for connections to transducers and actions on trees.

\begin{definition}[Infinite wreath product of a transformation monoid]\label{def:R-inf} Let $R$ be a monoid acting on a set $X$ on the right.
For $n \geq 0$, we denote by $X^n$ the set of words over $X$ of length $n$. For  any infinite sequence of functions $(F^n \colon X^{n-1} \to R)_{n \geq 1}$, we recursively define a function $F \colon X^* \to X^*$ by
\[\eps F := \eps, \text{ and } (x \cdot \q) F := x (\q F^n) \cdot (\q F) \text{ for any } x \in X, \q \in X^{n-1}, n \geq 1\]
and we say that the sequence of functions $F^n$ is an \emph{$R$-definition} of $F$. More explicitly, for $(q_n,\dots,q_1) \in X^*$,
\[
(q_n,\dots,q_1)F = (q_n((q_{n-1},\dots,q_1)F^n), \dots, q_2 (q_1F^2), q_1 (\eps F^1)).
\]

We define the \emph{infinite wreath product} $R^\infty$ to be the set of functions $F \colon X^* \to X^*$ that admit an $R$-definition.
\end{definition}

Note that any function $F$ in $R^\infty$ is length-preserving, and that, for any $\q = (q_n,\dots,q_1) \in X^*$, the $i^\th$ coordinate of $\q F$ only depends on the prefix $(q_i,\dots,q_1)$.

\begin{lemma}\label{lem:Rinf-monoid}
The set $R^\infty$ is a monoid of transformations on $X^*$.
\end{lemma}
\begin{proof}
The identity function is defined by letting $I^n$ be the function with constant value $1_R$ for every $n \geq 1$. If $(F^n)$ and $(G^n)$ are $R$-definitions of $F$ and $G$, respectively, then we define, for every $n \geq 1$, the function $H^n \colon X^{n-1} \to R$ by $\q H^n := (\q F^n) (\q FG^n)$ (using that $F$ is length-preserving). Then $(H^n)$ is an $R$-definition of $FG$; indeed, for any $x \in X$ and $\q \in X^*$, we have
\[ (x \cdot \q)FG = ( x (\q F^n) \cdot \q F ) G = x (\q F^n) (\q FG^n) \cdot \q F G = (x \q H^n) \cdot \q FG.\qedhere\]
\end{proof}

\begin{definition}[Shift action]\label{def:shift-action}
Let $F \in R^\infty$.
For any $\a \in X^*$, we denote by $F_\a \colon X^* \to X^*$ the function uniquely defined by the condition
\begin{equation}\label{eq:Fa-definition}
(\b \cdot \a)F = \b F_\a \cdot \a F \text{ for all } \b \in X^*.
\end{equation}
In words, $F_\a \colon X^* \to X^*$ sends any $\b \in X^*$ to the last $|\b|$ letters of the word $(\b \cdot \a)F$. We shall see momentarily that $F_\a\in R^\infty$.  Note that this is an (anti-)action of $X^*$ on $R^\infty$: $F_{\eps} = F$, and for any $\a, \a' \in X^*$, we have $(F_\a)_{\a'} = F_{\a' \cdot \a}$. (This is written as a left action in~\cite{HRS2010AP}.)

A subsemigroup $U$ of $R^\infty$ is \emph{self-similar} if, for any $F \in U$ and $a \in X$, we have $F_a \in U$. Note that, if $U$ is self-similar, then in fact $F_\a \in U$ for any $F \in U$ and $\a \in X^*$, as is easily shown by induction on the length of $\a$.

If $(F^n)$ $R$-defines $F$, then we write $\sigma_F := \eps F^1$, so that (\ref{eq:Fa-definition}) in particular gives, for any $\q = (q_n,\dots,q_1) \in X^+$,
\begin{equation}\label{eq:sigmaF-definition}
(q_n,\dots,q_1)F = (q_n,\dots,q_2) F_{q_1} \cdot q_1 \sigma_F.
\end{equation}
\end{definition}

\begin{lemma}\label{lem:Rinf-selfsimilar}
The monoid $R^\infty$ is self-similar.
\end{lemma}
\begin{proof}
Let $F \in R^\infty$ and $a \in X$. We will give an $R$-definition $(G^n)$ of $F_a$. For $n \geq 1$, $\c \in X^{n-1}$, define $\c G^n := (\c a)F^{n+1}$. Let $G \colon X^* \to X^*$ be the function with $R$-definition $(G^n)$. We show by induction on the length of $\b$ that $(\b \cdot a)F = \b G \cdot a F$. For $\b = \eps$, this is clear. Now assume that $(\c \cdot a)F = \c G \cdot a F$ for all words $\c$ with $|\c|=n-1$, and let $\b = b \cdot \c$, with $b \in X$. Then
\begin{align*}
(\b \cdot a) F = (b \cdot \c \cdot a) F &= b(\c \cdot a)F^{n+1} \cdot (\c \cdot a)F \\
&= b (\c G^n) \cdot \c G \cdot a F \\
&= (b \cdot \c) G \cdot a F = \b G \cdot a F.\qedhere
\end{align*}
\end{proof}

\subsection*{A monoid action on $S$ and the Zeiger property}
Recall that we continue to fix a semigroup $T$ and that we denote by $S$ the semigroup $\Sat_{\bHbar}(T\eta)$. %
We will apply the infinite wreath product construction to a particular monoid $R$ acting on (the underlying set of) $S$.
\begin{definition}\label{def:R}
Denote by $R$ the set of all functions $r \colon S \to S$ such that
\begin{enumerate}[label={(\roman*)}]
	\item $qr \leqR q$ for all $q \in S$,
	\item if $q  \sL  q'$ then $qr  \sL  q'r$,
	\item there exists $q_r \in S^I$ such that for any $q \in S$, if $qr \sR q$, then $qr = qq_r$.
\end{enumerate}
\end{definition}
\begin{proposition}\label{prop:R-monoid}
Let $R$ be the set of functions defined in Definition~\ref{def:R}.
\begin{enumerate}
\item The set $R$ is a transformation monoid on the set $S$.
\item The monoid $R$ contains any blowup operator $b$ on $S$.
\item For every $m \in S^I$, the right multiplication map, $r_m \colon q \mapsto qm$, is in $R$.
\end{enumerate}
\end{proposition}
\begin{proof}
(1) Suppose that $r, r' \in R$. The composite $rr'$ clearly satisfies (i) and (ii) in Def.~\ref{def:R} because $r$ and $r'$ do. For (iii), if $qrr' \sR q$, since we have $qrr' \leqR qr \leqR q \leqR qrr'$, so that all of these elements are $\R$-equivalent. In particular, $qr = qq_r$, and $qrr' = qrq_{r'}$, so that $qrr' = qq_rq_{r'}$. Thus, we may define $q_{rr'} := q_rq_{r'}$.

(2) Let $b$ be a blowup operator on $S$. We verify (i) -- (iii) in Def.~\ref{def:R}. (i) Since $qb = qm_{L_q}$, we clearly have $qb \leqR q$. (ii) If $q  \sL  q'$, then $qb = qm_{L_q}  \sL  q'm_{L_q} = q'm_{L_{q'}} = q'b$. (iii) By Lemma~\ref{lem:preblowup-L}(5), if $qb \sR q$, then $qb = q$, so we can take $q_b := I$.

(3) Properties (i) and (ii) are obvious, and in (iii) we can take $q_{r_m} := m$.
\end{proof}

We will now use the infinite wreath product $R^\infty$ to define, in Definition~\ref{def:SH} below, a semigroup $\SHbar$ acting on $\sFbar_\bH$, and from there a semigroup $S^\bH$ acting on $\sF_\bH$. We first isolate an important property of certain transformations in $R^\infty$, which was also used in \cite{HRS2010AP}, inspired by Zeiger's proof of the Krohn-Rhodes Theorem \cite{Zeiger1}.

\begin{definition}\label{def:pre-Zeiger}
Let $F \in R^\infty$. We will say $F$ has the \emph{pre-Zeiger property} if, for any $\q = (q_n, \dots, q_1) \in S^*$ with $|\q| \geq 2$ and $(s_n,s_{n-1},\dots,s_1) := \q F$ such that $q_{n-1} \sR s_{n-1}$, and $q_n \sR s_n$, there exists $t \in S^I$ such that $q_{n-1}t = s_{n-1}$ and $q_n t = s_n$.
\end{definition}

Recall that in Section~\ref{sec:blowup-flow}, we fixed a blowup operator $b$ on $S = \Sat_{\bHbar}(T\eta)$ and a multiplier $m\colon S/{\sL}\to S^I$, which we used in Def.~\ref{def:diagonal-and-Bhat} to define an operator $B$ on $S^*$.

\begin{lemma}\label{lem:Bhat-Rinf}
The operators $\Delta_s$, for every $s \in S^I$, and $B$ are in $R^\infty$ and have the pre-Zeiger property.
\end{lemma}
\begin{proof}
The sequence of functions $S^{n-1} \to R$, $n \geq 1$, which have constant value $r_s$ gives an $R$-definition of $\Delta_s$, and the pre-Zeiger property is trivially satisfied by $\Delta_s$.

We give an $R$-definition of $B$. Define $\eps B^1 := b$, and recursively, for $n \geq 2$, $(q_{n-1},\dots,q_1) \in S^{n-1}$, define
\[ (q_{n-1},\dots,q_1)B^n := r_{m_{L_{q_1}}} ((q_{n-1}, \dots, q_2)\Delta_{m_{L_{q_1}}})B^{n-1}.\]
Denote by $\tilde{B}$ the function with $R$-definition $(B^n)_{n \geq 1}$. We will show that, for every $\q \in S^*$ and $s \in S$,
\begin{equation} \label{eq:tildeB}
(\q \cdot s)\tilde{B} = \q \Delta_{m_{L_s}} \tilde{B} \cdot sb,
\end{equation}
which will establish that $\tilde{B} = B$, by Definition~\ref{def:diagonal-and-Bhat}. We prove (\ref{eq:tildeB}) by induction on the length of $\q$. If $\q = \eps$, then $s\tilde{B} = s (\eps B^1) = sb$, as required. For $n \geq 1$, we have
\begin{align*}
(\q \cdot s) \tilde{B} &= q_n (q_{n-1},\dots,q_1,s) B^{n+1} \cdot (q_{n-1}, \dots, q_1, s)\tilde{B} &\text{(Def. of $\tilde{B}$)} \\
&= q_n (q_{n-1},\dots,q_1,s)B^{n+1} \cdot (q_{n-1},\dots,q_1)\Delta_{m_{L_s}}\tilde{B} \cdot sb &\text{(Ind. hyp.)} \\
&= q_n m_{L_s} ((q_{n-1},\dots,q_1)\Delta_{m_{L_s}})B^n \cdot (q_{n-1},\dots,q_1)\Delta_{m_{L_s}}\tilde{B} \cdot sb  &\text{(Def. of $B^{n+1}$)} \\
&= \q \Delta_{m_{L_s}} \tilde{B} \cdot sb &\text{(Def. of $\tilde{B}$)}.
\end{align*}
We prove the pre-Zeiger property for $B$ by induction on the length of $\q = (q_n,\dots,q_1)$, $n \geq 2$. When $n = 2$, we have $(q_2,q_1)B = (q_2m_Lb,q_1b)$, where $L$ denotes the $\L$-class of ${q_1}$. Suppose that $q_2 \sR q_2m_Lb$ and $q_1 \sR q_1b$. Then $q_1b = q_1m_L$. Also,
\[ q_2 \geq_{\R} q_2m_L \geq_{\R} q_2m_Lb \sR q_2,\]
so that $q_2m_L \sR q_2m_Lb$, and by Lemma~\ref{lem:preblowup-L}(5), $q_2m_Lb = q_2m_L$. Thus, choosing $t := m_L$ gives $q_1t = q_1b$ and $q_2t = q_2m_Lb$. Now suppose that $n > 2$ and write $(s_n,\dots,s_1) := \q B$, with $q_i \sR s_i$ for $i = n, n-1$. Writing $L := L_{q_1}$, by definition of $B$, we have $(s_n,\dots,s_2) = (q_nm_{L},\dots,q_2m_{L})B$. For $i = n, n - 1$, we have that $s_i \leqR q_im_L$ since $B$ is in $R^\infty$, and, on the other hand, we have $q_im_L \leqR q_i \sR s_i$, so $s_i \sR q_im_L$. Thus, the induction hypothesis applies to the word $(q_nm_{L},\dots,q_2m_{L})$, and we can pick $t' \in S^I$ such that $s_i = q_im_Lt'$ for $i = n, n-1$. Choosing $t := m_Lt'$ yields the result.
\end{proof}

\begin{lemma}\label{lem:pre-Zeiger-semigroup}
The set of transformations $F \in R^\infty$ which satisfy the pre-Zeiger property is a self-similar subsemigroup of $R^\infty$.
\end{lemma}
\begin{proof}
Suppose that $F, G \in R^\infty$ have the pre-Zeiger property. Let $\q = (q_n,\dots,q_1)$ be in $S^*$ with $n \geq 2$ and suppose that $(s_n,\dots,s_1) := \q FG$ is such that $q_n \sR s_n$ and $q_{n-1} \sR s_{n-1}$. Write $(t_n,\dots,t_1) := \q F$. Note that, for any $1 \leq i \leq n$, $s_i \leqR t_i \leqR q_i$, using Def~\ref{def:R}(i), since $F$ and $G$ are in $R^\infty$. Therefore, for $i = n, n-1$, since $q_i \sR s_i$, we have that $s_i \sR t_i \sR q_i$. By the pre-Zeiger property for $G$, pick $u \in S^I$ such that $t_i u = s_i$ for $i = n, n-1$, and by the pre-Zeiger property for $F$, pick $v \in S^I$ such that $q_i v = t_i$ for $i = n, n - 1$. Then $q_i vu = s_i$ for $i = n,n-1$, as required.
For the self-similarity, assume $F$ has the pre-Zeiger property, and suppose that, for $a \in S$, we have $(s_n,\dots,s_1) = \q F_a$ and $q_n  \sR  s_n$,  $q_{n-1}  \sR  s_{n-1}$. By (\ref{eq:Fa-definition}) in Def.~\ref{def:shift-action}, $(\q \cdot a)F = (s_n,s_{n-1},\dots,s_1) \cdot aF$, so the pre-Zeiger property for $F$ applies.
\end{proof}

When a pre-Zeiger transformation in $R^\infty$ is applied to an $\L$-chain in $S$, we obtain the following property, which will be crucial in the proof of Theorem~\ref{thm:SH-in-Hbar}.
\begin{lemma}[Zeiger Property]\label{lem:Zeiger}
Suppose that $F \in R^\infty$ has the pre-Zeiger property. Then, for any $\q = (q_n,q_{n-1},\dots,q_1) \in \sFbar$ with $n \geq 2$, and $(s_n,s_{n-1},\dots,s_1) := \q F$ such that $q_{n-1} = s_{n-1}$ and $q_n  \sR  s_n$, we have $q_n = s_n$.
\end{lemma}
\begin{proof}
By the pre-Zeiger property of $F$, pick $t \in S^I$ such that $q_{n-1}t = s_{n-1} = q_{n-1}$ and $q_n t = s_n$. Since $q_n \leqL q_{n-1}$, pick $u \in S^I$ such that $q_n = uq_{n-1}$. Then $s_n = q_n t = uq_{n-1}t = uq_{n-1} = q_n$, as required.
\end{proof}

\subsection*{A semigroup containing $\T_\A$}
We are now ready to define the semigroup $S^\bH$ that will contain $\T_\A$, the transition semigroup of the automaton $\A$. The semigroup $S^\bH$ will be built from a semigroup $\SHbar$, whose elements are asynchronous transducers, in the following sense.
\begin{definition}\label{def:SH}
For any $\f \in \sFbar \setminus \{\eps\}$ and $F \in R^\infty$, define a function $\fbar \colon S^* \to S^*$ by
\[\q \fbar := \q F \cdot \f.\]
We call any function $\fbar \colon S^* \to S^*$ which arises in this way an \emph{asynchronous $R$-transducer} and we call the non-empty $\L$-chain $\f$ the \emph{asynchronous part} of $\fbar$, and $F$ the \emph{synchronous part} of $\fbar$. Note that, if $\fbar$ is an asynchronous $R$-transducer, then $F$ and $\f$ are uniquely determined by $\fbar$: indeed, $\f = \eps \fbar$ and, for any $\q \in S^*$, $\q F$ is the length $|\q|$ suffix of $\q \fbar$.

Let $\SHbar$ be the set of asynchronous $R$-transducers $\fbar \colon S^* \to S^*$ such that
\begin{enumerate}[label={(\roman*)}]
\item the sets $\sFbar$ and $\sFbar_\bH$ are invariant under $\fbar$;
\item the restriction of $\fbar$ to $\sFbar$ respects $\rho$;
\item the synchronous part $F \in R^\infty$ has the pre-Zeiger property.
\end{enumerate}
Note that, in particular, (i) implies that $\f = \eps\fbar$ is in $\sFbar_\bH$. Also note that, for any function $\fbar$ in $\SHbar$, the composite $\fbar\rho$ maps $\sF_\bH$ into $\sF_\bH$.
We define
\[ S^\bH := \{ f \colon \sF_\bH \to \sF_\bH \mid f = \fbar|_{\sF_\bH}\rho \text{ for some } \fbar \in \SHbar\},\]
and we denote by $\pi \colon \SHbar \to S^\bH$ the function which sends $\fbar \in \SHbar$ to $\fbar\pi := \fbar|_{\sF_\bH}\rho$.

For $f \in S^\bH$, an element $\fbar \in \SHbar$ such that $\fbar\pi = f$ is called a \emph{representative} of $f$.
\end{definition}

\begin{notation}
In what follows, we will use the notational convention that whenever $\fbar$, $\gbar$, $\dots$ are asynchronous $R$-transducers, the capital letters $F, G, \dots$ denote their synchronous parts, and the boldface lowercase letters $\f, \g, \dots$ denote their asynchronous parts. Also, for $f, g, \dots$ in $\SH$, we will denote by $\fbar, \gbar, \dots$ an arbitrary representative of it in $\SHbar$.
\end{notation}

It is not immediate from Definition~\ref{def:SH} that $\SHbar$ and $S^\bH$ are indeed semigroups; we will prove this now. The following lemma shows how asynchronous transducers compose, and will also be used repeatedly in the proof of Theorem~\ref{thm:SH-in-Hbar}.
\begin{lemma}\label{lem:compose-transducers}
For any $\fbar, \gbar \in \SHbar$, $\q \in S^*$, we have
\[ \q \fbar\gbar = \q FG_{\f} \cdot \f G \cdot \g.\]
\end{lemma}
\begin{proof}
Indeed,
\[ \q\fbar\gbar = (\q F \cdot \f)G \cdot \g = \q FG_\f \cdot \f G \cdot \g,\]
using (\ref{eq:Fa-definition}) for the second equality.
\end{proof}
\begin{proposition} \label{prop:SH-semigroup}
The set $\SHbar$ is a transformation semigroup on $S^*$, the set $S^\bH$ is a transformation semigroup on $\sF_\bH$, and the function $\pi \colon \SHbar \to S^\bH$ is a semigroup homomorphism.
\end{proposition}
\begin{proof}
To see that $\SHbar$ is a semigroup, let $\fbar, \gbar \in \SHbar$, and define $H := FG_\f$ and $\mathbf{h} := \f G \cdot \g$. By Lemma~\ref{lem:compose-transducers}, $\q\fbar\gbar = \q H \cdot \mathbf{h}$ for any $\q \in S^*$.
By Lemmas~\ref{lem:Rinf-monoid},~\ref{lem:Rinf-selfsimilar} and~\ref{lem:pre-Zeiger-semigroup}, $H$ is in $R^\infty$ and has the pre-Zeiger property. Also, $\mathbf{h} \in \sFbar \setminus \{\eps\}$ because it is a non-empty prefix of the $\L$-chain $\q \fbar \gbar$. Moreover, conditions (i) and (ii) in Def.~\ref{def:SH} also clearly hold for $\fbar \gbar$. Thus, $\SHbar$ is a semigroup. Now, if $f, g \in S^\bH$ and $\fbar$, $\gbar$ are representatives of $f$, then, for any $\q \in \sF_\bH$,
\[ \q fg = \q \fbar \rho \gbar \rho = \q \fbar \gbar \rho,\]
so $\fbar\gbar$ is a representative of $fg$. This shows that $S^\bH$ is a semigroup, and that $\pi$ is a homomorphism.
\end{proof}

\begin{proposition}\label{prop:SH-contains-taus}
The transition semigroup $\T_\A$ of the automaton $\A$ is a subsemigroup of the semigroup $S^\bH$.
\end{proposition}
\begin{proof}
We prove that, for every $t \in T$, $\taubar_t$ is in $\SHbar$. Since by definition $\tau_t = \taubar_t|_{\sFbar}\rho$, it will then follow that $\tau_t$ is in $S^\bH$, as required. %
Recall that we are viewing $T$ as a subsemigroup of $2^T$.  From the definitions of $\taubar_t$ and $B$, we see that, for any $\q \in S^*$,
\[\q \taubar_t = (\q\Delta_{t} \cdot t)B = \q\Delta_{t}\Delta_{m_{L_{t}}}B \cdot tb=\q T_t\cdot \bm{\tau}_t\] where
$T_t := \Delta_{t}\Delta_{m_{L_{t}}}B$ and $\bm{\tau}_t := tb$. By Lemmas~\ref{lem:Rinf-monoid},~\ref{lem:Bhat-Rinf}~and~\ref{lem:pre-Zeiger-semigroup}, $T_t$ is in $R^\infty$ and has the pre-Zeiger property. %
Properties (i) and (ii) in Def.~\ref{def:SH} hold by Lemma~\ref{lem:taubar-on-states}. %
\end{proof}

\subsection*{The semigroup is in $\bHbar$} We will prove in Theorem~\ref{thm:SH-in-Hbar} that the semigroup $S^\bH$, which contains $\T_\A$, is in $\bHbar$. This theorem generalizes \cite[Thm.~4.15]{HRS2010AP} by making more serious use of the Sch\"utzenberger group.

The first lemma contains some simple observations about how the `last letter' operation $\omega$ interacts with elements of $\SHbar$ and $\SH$.
\begin{lemma}\label{lem:omega-and-SH}
Let $f \in \SH$ and let $\fbar$ be an arbitrary representative of $f$ with synchronous part $F$. Then
\begin{enumerate}
	\item for any $\q \in \sFbar \setminus \{\eps\}$, $\q\fbar\omega = \q F \omega \leqR \q\omega$.
	\item for any $\q \in \sF_\bH$, $\q\fbar\omega = \q f \omega$,
\end{enumerate}
\end{lemma}
\begin{proof}
(1) Note that, since $\q \neq \eps$, $\q\fbar = \q F \cdot \f$ with $\q F \neq \eps$, so $\q\fbar\omega = \q F \omega$. Since $F \in R^\infty$, each coordinate of $\q F$ is ${\R}$-below the corresponding coordinate of $\q$, so, in particular, $\q F \omega \leqR \q \omega$. (2) By Lemma~\ref{lem:rho-properties}(2), $\q\fbar\omega = \q\fbar\rho\omega$, and the latter is $\q f \omega$.
\end{proof}

The following lemma identifies special properties of the last elements of $\q f$ and $\q g$, when $f$ and $g$ belong to a subgroup of $\SH$.
\begin{lemma}\label{lem:omega-eps}
Let $G$ be a subgroup of $\SH$ with unit $u$. Write $H$ for the $\sH$-class of $\eps u \omega$, and denote by $\u'$ the prefix of $\eps u$ obtained by deleting the last letter, i.e., $\eps u = \eps u \omega \cdot \u'$.  Then, for any $f, g \in G$:
\begin{enumerate}
\item for any $\q \in \sF_\bH$, the elements $\q f \omega$ and $\q g \omega$ are $\R$-equivalent;
\item the elements $\eps f \omega$ and $\eps g \omega$ are $\sH$-equivalent, and lie in $H$;
\item $\eps g = \eps g \omega \cdot \u'$.
\end{enumerate}
\end{lemma}
\begin{proof}
Let $\fbar, \gbar \in \SHbar$ be representatives of $f$ and $g$, respectively.

(1) Let $\q \in \sF_\bH$. Pick a representative $\bar{h}$ of $g^{-1}f$. Then, since $\pi$ is a homomorphism, $\gbar \bar{h}$ is a representative of $f$. By Lemma~\ref{lem:omega-and-SH}(1), since $\q \gbar \neq \eps$, we have $\q \gbar \bar{h}\omega \leqR \q \gbar \omega$. We now get
\[ \q f \omega = \q \gbar \bar{h}\omega \leqR \q \gbar \omega = \q g \omega,\]
where the first and last equality follow from Lemma~\ref{lem:omega-and-SH}(2). By symmetry, we conclude that $\q f \omega$ and $\q g \omega$ are ${\R}$-equivalent.

(2) Pick a representative $\bar{h}$ of $gf^{-1}$, so that $\bar{h}\bar{f}$ is a representative of $g$. Then
\[ \eps g = \eps \bar{h} \bar{f} \rho = (\eps \bar{h} F \cdot \f)\rho.\]
Thus, the $\L$-chain $\eps g$, in particular, contains a letter which is $\L$-equivalent to $\f \omega = \eps \fbar \omega$. Therefore, $\eps g \omega \leqL \eps \fbar \omega = \eps f \omega$, using Lemma~\ref{lem:omega-and-SH}(2). By symmetry, it follows that the two elements are $\L$-equivalent. By item (1) applied to $\q = \eps$, they are also $\R$-equivalent, and hence $\sH$-equivalent. Since in particular $u \in G$, we conclude that $\eps f \omega \sH  \eps u \omega$, so $\eps f \omega \in H$.

(3) Choose a representative $\bar{u}$ of $u$. Since $gu = g$, we have
\[\eps g = \eps gu = \eps \gbar \ubar \rho = (\eps\gbar U \cdot \u)\rho.\]
By item (2), $\u\omega = \eps u \omega$ is in particular $\L$-equivalent to $\eps g \omega$. Therefore, the $\L$-chain $\eps\gbar U \cdot \u$ reduces to $\eps g \omega \cdot \u'$, and the latter is a strict $\sL$-chain, because $\eps g \omega$ is ${\L}$-equivalent to $\eps u \omega$, and $\eps u = \eps u \omega \cdot \u'$ is a strict $\sL$-chain.
\end{proof}

We use what we have proved so far to construct a homomorphism from any subgroup of $\SH$ to a Sch\"utzenberger group. This will be a key ingredient in the proof of Theorem~\ref{thm:SH-in-Hbar}.

\begin{proposition}\label{prop:hom-to-Gamma}
Let $G$ be a subgroup of $\SH$ with unit $u$, and denote by $H$ the $\sH$-class of $\eps u \omega$. There exists a homomorphism $\phi \colon G \to \Gamma_R(H)$ such that, for every $f, g \in G$, $\eps g \omega (f\phi) = \eps gf \omega$.
\end{proposition}
\begin{proof}
Let $\u'$ be defined as in the statement of Lemma~\ref{lem:omega-eps} and let $\bar{f}$ be a representative of $f$. Write $\sigma_f$ for the element of $R$ given by the restriction of $F_{\u'}$ to one-letter words, and pick $q_{\sigma_f} \in S^I$ as in Definition~\ref{def:R}(iii).
For any $g \in G$, we have
\begin{align*}
\eps gf \omega &= (\eps g \omega \cdot \u')f \omega  &\text{(by Lemma~\ref{lem:omega-eps}(3))} \\
&= (\eps g \omega \cdot \u')\fbar \omega &\text{(by Lemma~\ref{lem:omega-and-SH}(2))} \\
&= ((\eps g \omega \cdot \u')F \cdot \f) \omega &\text{(Definition of $\fbar$)}\\
&= \eps g \omega F_{\u'} &\text{(Definition~\ref{def:shift-action})} \\
&= \eps g \omega \sigma_f  &\text{(Definition of $\sigma_f$)} \\
&= \eps g \omega q_{\sigma_f} &\text{(Definition~\ref{def:R}(iii)),}
\end{align*}
where the last equality uses that, by Lemma~\ref{lem:omega-eps}(1), $\eps g \omega$ and $\eps gf \omega = \eps g \omega \sigma_f$ are {$\R$}-equivalent.
In particular, applying this equality to $g := u$, we have $\eps u \omega q_{\sigma_f} = \eps uf \omega = \eps f \omega$, which is $\sH$-equivalent to $\eps u \omega$ by Lemma~\ref{lem:omega-eps}(2). Thus, the element $q_{\sigma_f}$ lies in the stabilizer of the $\sH$-class $H$, by Lemma~\ref{lem:one-for-all-right}. We define $f\phi := r_{q_{\sigma_f}}$, right multiplication by $q_{\sigma_f}$, which is an element of the right Sch\"utzenberger group $\Gamma_R(H)$ which has the stated property.
Moreover, if $f, f' \in G$, then
\[\eps u \omega (f \phi)(f' \phi) = \eps f \omega (f'\phi) = \eps ff' \omega = \eps u \omega (ff' \phi),\]
so that $(f\phi)(f'\phi) = ff' \phi$, because the action of $\Gamma_R(H)$ on $H$ has trivial point stabilizers (Lemma~\ref{lem:trivial-point-stabilizers}).
\end{proof}

\begin{theorem}\label{thm:SH-in-Hbar}
The semigroup $S^\bH$ is in $\bHbar$.
\end{theorem}
\begin{proof}
Let $G$ be a subgroup of $S^\bH$. Denote the unit of $G$ by $u$. We will prove that, for any $f \in K_\bH(G)$, and
\begin{equation}\label{eq:f-equals-u}
\text{for all } \q \in \sF_\bH, \quad \q f = \q u.
\end{equation}
From (\ref{eq:f-equals-u}), it will follow that $K_\bH(G)$ is trivial, so that $G$ is in $\bH$, as required. As usual, choose representatives $\bar{f}$ and $\bar{u}$ of $f$ and $u$, respectively.

Let $f \in K_\bH(G)$ be arbitrary. The proof of (\ref{eq:f-equals-u}) is by induction on the length $|\q|$ of $\q$. We first need to establish separately the cases $|\q| = 0$ and $|\q| = 1$.

{\it Case 1: $\q = \eps$.} Let $\phi \colon G \to \Gamma_R(H)$ be the homomorphism of Proposition~\ref{prop:hom-to-Gamma}. Since $\eps u \omega$ is an $\bH$-element, $\Gamma_R(H)$ lies in $\bH$. Therefore, since $f \in K_\bH(G)$, we have $f \in \ker \phi$, i.e., $f\phi$ acts on $H$ as the identity. In particular,
\[\eps f \omega = \eps u f \omega = \eps u \omega (f\phi) = \eps u \omega.\]
Using Lemma~\ref{lem:omega-eps}(3), from this we obtain
\[ \eps f = \eps f \omega \cdot \u' = \eps u \omega \cdot \u' = \eps u\]
where we retain the notation of that lemma.

For the next two cases, we will repeatedly use that, for any $\q \in \sF_\bH$,
\begin{equation}\label{eq:uf-trick}
 \q f = \q uf = \q\ubar\fbar\rho = (\q U F_\u \cdot \u F \cdot \f)\rho,
 \end{equation}
where the second equality holds because $\ubar$ and $\fbar$ respect $\rho$, and the third equality holds by Lemma~\ref{lem:compose-transducers}.

{\it Case 2: $|\q| = 1$.} Write $q$ for the single letter in $\q$. We aim to apply the Zeiger Property (Lemma~\ref{lem:Zeiger}) for $F$. Write $(q_n,\dots,q_1) := q\ubar$, and $(s_n,\dots,s_1) := q\ubar F$. We have
$q\ubar = qU \cdot \u$, so  $q_n = qU$ and $q_{n-1} = \u\omega$. Also, using (\ref{eq:Fa-definition}),
\[q\ubar F = (qU \cdot \u)F = qUF_\u \cdot \u F,\]
and so $s_n = qUF_\u$, $s_{n-1} = \u F \omega$.

Now, by Lemma~\ref{lem:omega-eps}(1), $q_n = qu\omega$ is $\R$-equivalent to $qf\omega$, and the latter is in fact equal to $q U F_\u = s_n$, by (\ref{eq:uf-trick}).
Also, $q_{n-1} = \u\omega = \u F \omega = s_{n-1}$, because
\[ \u\omega = \eps u \omega = \eps f \omega = \eps uf\omega = \u f \omega = \u F \omega,\]
where we have used Case 1 and Lemma~\ref{lem:omega-and-SH}. Thus, the conditions of Lemma~\ref{lem:Zeiger} are fulfilled for $F$ applied to $q\ubar$. We conclude that
\begin{equation}\label{eq:length1-part1}
qU= q_n = s_n = qUF_\u.
\end{equation}
Also, using Case 1 again, we have
\begin{equation}\label{eq:length1-part2}
\u\rho = \eps u = \eps f = \eps uf = \u \rho \fbar \rho = \u \fbar \rho = (\u F \cdot \f)\rho.
\end{equation}
We now compute:
\begin{align*}
qf &= (qUF_{\u} \cdot \u F \cdot \f)\rho &\text{(using (\ref{eq:uf-trick}))}\\
&= (qU \cdot \u F \cdot \f)\rho &\text{(using (\ref{eq:length1-part1}))}\\
&= (qU \cdot (\u F \cdot \f)\rho)\rho &\text{(Lemma~\ref{lem:rho-properties})}\\
&= (qU \cdot \u\rho)\rho &\text{(using (\ref{eq:length1-part2}))}\\
&= (qU \cdot \u)\rho = q\ubar\rho = qu. &\text{(Lemma~\ref{lem:rho-properties})}
\end{align*}

{\it Case 3: $|\q| > 1$.} Write $\q = q \cdot \q'$, so that $|\q'| = |\q| - 1 > 0$, and, by the induction hypothesis, $\q'f = \q'u$. We aim to apply the Zeiger Property for $F_{\u}$. Write $(t_n,\dots,t_1) := \q U$ and $(v_n,\dots,v_1) := \q U F_{\u}$. Using (\ref{eq:Fa-definition}), we compute:
\begin{equation} \label{eq:qU}
\q U = (q \cdot \q')U = qU_{\q'} \cdot \q'U,
\end{equation}
and, hence,
\begin{equation} \label{eq:qUFu}
\q U F_\u = (qU_{\q'} \cdot \q'U)F_{\u} = qU_{\q'}F_{\q'U \cdot \u} \cdot \q'UF_{\u}.
\end{equation}
From (\ref{eq:qU}), $t_n = qU_{\q'} = \q U \omega = \q u \omega$ by Lemma~\ref{lem:omega-and-SH}, and $t_{n-1} = \q'U\omega = \q'u\omega$, again by Lemma~\ref{lem:omega-and-SH}, since $\q' \neq \eps$. From (\ref{eq:uf-trick}) and Lemma~\ref{lem:omega-and-SH}, we have that $v_n = \q U F_{\u}\omega = \q f \omega$. From (\ref{eq:qUFu}) and (\ref{eq:uf-trick}) applied to $\q'$ we have that $v_{n-1} = \q'UF_{\u}\omega = \q'f\omega$. By the induction hypothesis, $t_{n-1} = \q'u\omega = \q'f\omega = v_{n-1}$, and by Lemma~\ref{lem:omega-eps}(1), $t_n  \sR  v_n$. Thus, the conditions of Lemma~\ref{lem:Zeiger} are fulfilled for $F_{\u}$ applied to $\q U$, and we conclude that
\begin{equation}\label{eq:Zeiger-result}
qU_{\q'}F_{\q'U \cdot \u}=\q UF_{\u} =v_n=t_n= qU_{\q'}.
\end{equation}
We now compute:
\begin{align*}
\q f &= (\q U F_\u \cdot \u F \cdot \f)\rho &\text{(using (\ref{eq:uf-trick}))}\\
&= (qU_{\q'}F_{\q'U \cdot \u} \cdot \q'UF_{\u} \cdot \u F \cdot \f)\rho &\text{(using (\ref{eq:qUFu}))}\\
&= (qU_{\q'} \cdot \q'UF_{\u} \cdot \u F \cdot \f)\rho &\text{(using (\ref{eq:Zeiger-result}))}\\
&= (qU_{\q'} \cdot (\q'UF_{\u} \cdot \u F \cdot \f)\rho)\rho &\text{(Lemma~\ref{lem:rho-properties})}\\
&= (qU_{\q'} \cdot \q'f)\rho &\text{(using (\ref{eq:uf-trick}))}\\
&= (qU_{\q'} \cdot \q'u)\rho &\text{(induction hypothesis)}\\
&= (qU_{\q'} \cdot \q'U \cdot \u)\rho &\text{(Definition of $\ubar$ and Lemma~\ref{lem:rho-properties})}\\
&= (\q U \cdot \u)\rho = \q u &\text{(using (\ref{eq:Fa-definition}))},
\end{align*}
as required.
\end{proof}

\begin{corollary}\label{cor:TA-in-Hbar}
The transition semigroup $\T_\A$ of the automaton $\A$ is in $\bHbar$.
\end{corollary}
\begin{proof}
Immediate from Proposition~\ref{prop:SH-contains-taus} and Theorem~\ref{thm:SH-in-Hbar}.
\end{proof}

\section{Proof of Main Theorem}\label{sec:mainproof}
\begin{proof}[Proof of Theorem~\ref{thm:main}]
Let $T$ be a finite semigroup. Note that $\PL_{\bHbar}(T)$ is $\bHbar$-saturated: conditions (i) and (ii) in Def.~\ref{def:saturated} are true for the pointlike sets with respect to any variety, and (iii) follows from Propositions~\ref{prop:pointlike-monad} and \ref{prop:H-kernel-pointlike}. Thus, since $\PL_{\bHbar}(T)$ contains $T\eta$, we have $\Sat_{\bHbar}(T\eta) \subseteq \PL_{\bHbar}(T)$.
We now prove the converse. Let $\A = (\sF_\bH,T,\tau,i)$ be the automaton and $\Phi$ the $T$-flow defined in Definition~\ref{def:tau}. Corollary~\ref{cor:TA-in-Hbar} shows that $\A$ is an $\bHbar$-automaton. Now let $X\neq \emptyset$ be an $\bHbar$-pointlike subset of $T$. By Proposition~\ref{prop:pointlike-flow}, $X \subseteq \q\Phi$ for some $\q \in \sF_\bH$. Since $X$ is non-empty and does not contain $I$, $\q \neq \eps$. Therefore, $X \subseteq \q\omega$, which is an element of $S = \Sat_{\bHbar}(T\eta)$. Since $\Sat_{\bHbar}(T\eta)$ is downward closed, $X \in \Sat_{\bHbar}(T\eta)$, as required.
\end{proof}

\section{Alternative descriptions of the $\bHbar$-pointlikes and reducibility}
In this section, we assume familiarity with the theory of relatively free profinite semigroups and pseudoidentities; for more background, see, e.g.,~\cite{Almeida:book,RS2009}.  Let $A=\{a_1,\ldots, a_n\}$ be a finite alphabet.  If $\V$ is a variety of finite semigroups, then the free pro-$\V$ semigroup on $A$ is denoted by $\widehat{F}_{\V}(A)$.  However, we generally denote the free profinite semigroup on $A$ by  $\widehat{A^+}$.  Let $u\in \widehat{F}_{\V}(A)$.  If $S$ is a pro-$\V$ semigroup and $s_1,\ldots, s_n\in S$, then $u(s_1,\ldots, s_n)$ denotes the value of $u$ under the unique continuous homomorphism $\widehat{F}_{\V}(A)\to S$ sending $a_i$ to $s_i$.  Let us write $\pi_{\mathbf V}\colon \widehat{A^+}\to \widehat{F}_{\mathbf V}(A)$ for the canonical projection.

Let $e$ be an idempotent of the minimal ideal of $\widehat{A^+}$.  Then it was observed by Almeida and Volkov~\cite{AV03}, that $ea_1e,\ldots, ea_ne$ freely topologically generate a free profinite subgroup $G(e)$ of $\widehat{A^+}$, which is a closed subgroup of the maximal subgroup $G_e=e\widehat{A^+}e$.  Indeed, let $\mathbf G$ denote the variety of finite groups.  Then under the projection $\pi_{\mathbf G}\colon \widehat{A^+}\to \widehat{F}_{\mathbf G}(A)$, we have that $\pi_{\mathbf G}(e)=1$ and so $\pi_{\mathbf G}(ea_ie) = a_i$.  Thus $\pi_{\mathbf G}$ restricts to a continuous epimorphism of profinite groups $G(e)\to \widehat{F}_{\mathbf G}(A)$ which splits via $a_i\mapsto ea_ie$ by the universal property of $\widehat{F}_{\mathbf G}(A)$.  If $u\in \widehat{F}_{\mathbf G}(A)$, we will use $u_e$ as a shorthand for $u(ea_1e,\ldots,ea_ne)\in G(e)$, as we shall later evaluate $u_e$ in different finite semigroups.

We now wish to give a description of the $\bH$-kernel of a finite group coming from a basis of group pseudoidentities for $\bH$.

\begin{proposition}\label{p:Hkernel.verbal}
Let $\bH$ be a variety of finite groups and let $E$ be a collection of group pseudoidentities of the form $u=1$ defining $\bH$.  Let $G$ be a finite group.  Then $K_{\bH}(G)$ is the subgroup generated by all values of left hand sides of elements of $E$ in $G$.
\end{proposition}
\begin{proof}
Let $N$ be the subgroup generated by all values of left hand sides of elements of $E$ in $G$.    If $u=1$ belongs to $E$ and $x=u(g_1,\ldots, g_n)\in G$ is a value of $u$, then $xK_{\bH}(G) = u(g_1K_{\bH}(G),\ldots, g_nK_{\bH}(G))=K_{\bH}(G)$ as $G/K_{\bH}(G)\in \bH$ and so $x\in K_{\bH}(G)$.  Thus $N\subseteq K_{\bH}(G)$.    Note that $N$ is a normal subgroup because if $u(g_1,\ldots, g_n)$ is a value of $u$ and $g\in G$, then $gu(g_1,\ldots, g_n)g^{-1} = u(gg_1g^{-1},\ldots, gg_ng^{-1})$.  We claim that $G/N\in \bH$.  Indeed, if $u=1$ belongs to $E$, then $u(g_1N,\ldots, g_nN)= u(g_1,\ldots, g_n)N=N$ and so $G/N$ satisfies all of the pseudoidentities of $E$ and hence belongs to $\bH$. Thus $K_{\bH}(G)\subseteq N$.
\end{proof}

Proposition~\ref{p:Hkernel.verbal} allows us to provide a more compact description of the $\bHbar$-pointlikes that can often be more practical for computations.  Also the description of $\overline{\mathbf G_{\pi}}$-pointlikes, for a recursive set of primes $\pi$, given in~\cite{HRS2010AP} can be recovered in this way.

\begin{proposition}\label{p:gens}
Let $T$ be a finite semigroup and $S$ a subsemigroup of $2^T$ closed downward in the order.  Let $K$ be a subgroup of $S$ generated by a subset $\mathcal X$.  Then $\bigcup K\in S$ if and only if $\bigcup \langle X\rangle\in S$ for all $X\in \mathcal X$.
\end{proposition}
\begin{proof}
Clearly, if $X\in \mathcal X$, then $\bigcup\langle X\rangle \subseteq \bigcup K$ and so $\bigcup K\in S$ implies $\bigcup \langle X\rangle\in S$.  Assume now that $\bigcup \langle X\rangle\in S$ for all $X\in \mathcal X$. Since $K$ is finite, we can find $X_1,\ldots, X_n\in \mathcal X$ that generate $K$.  Let \[Y=\left(\bigcup \langle X_1\rangle \cdots \bigcup \langle X_n\rangle\right)^{\omega} = \bigcup \big(\langle X_1\rangle \cdots \langle X_n\rangle\big)^{\omega}\] where the second equality uses that $\bigcup\colon 2^{2^T}\to 2^{T}$ is a semigroup homomorphism.  Then $(\langle X_1\rangle \cdots \langle X_n\rangle)^{\omega}$ is a subsemigroup, and hence a subgroup, of $K$ containing $X_1,\ldots, X_n$ and thus is $K$.  We conclude that $\bigcup K=Y\in S$.
\end{proof}

We now give our alternative description of the $\bHbar$-pointlike sets.  For finitely based varieties of finite groups, it can give a more compact and computationally useful description (depending on the nature of the pseudoidentities).

\begin{theorem}\label{t:Hpointlikes}
Let $\bH$ be a variety of finite groups given by a basis of group pseudoidentities $E$. For $u=1$ in $E$, let $A_u$ be the alphabet of $u$ and fix an idempotent $e_u$ in the minimal ideal of $\widehat{A_u^+}$.  Put $u'=u_{e_u}$ for convenience.  Let $T$ be a finite semigroup.  Then the following are equivalent for a downward closed subsemigroup  $S$ of $2^T$.
\begin{enumerate}
  \item $S$ is $\bHbar$-saturated.
  \item If $u=1$ in $E$ (with $n$ variables) and $X_1,\ldots, X_n\in S$ belong to some subgroup $G$ of $S$, then $\bigcup\langle u(X_1,\ldots, X_n)\rangle \in S$.
  \item If $u=1$ in $E$ (with $n$ variables) and $X_1,\ldots, X_n\in S$, then \[\bigcup\langle u'(X_1,\ldots, X_n)\rangle \in S.\]
\end{enumerate}
\end{theorem}
\begin{proof}
By Proposition~\ref{p:Hkernel.verbal} the $\bH$-kernel of a subgroup $G$ of $S$ is generated by the values of left hand sides of elements of $E$. Since $S$ is downward closed, Proposition~\ref{p:gens} with $K=K_{\bH}(G)$ yields the equivalence of the first two items.  Since $u'$ maps to $u$ under the natural projection $\pi_{\mathbf G}\colon \widehat{A_u^+}\to \widehat{F}_{\mathbf G}(A_u)$, if $X_1,\ldots, X_n\in S$ belong to some subgroup $G$ of $S$ then $u(X_1,\ldots, X_n) = u'(X_1,\ldots, X_n)$. Thus the third item implies the second. Suppose that $X_1,\ldots, X_n\in S$ and put $e=e_u(X_1,\ldots, X_n)$.  Then $eX_1e,\ldots, eX_ne$ belong to the maximal subgroup $G_e$ of $S$ and $u'(X_1,\ldots,X_n) = u(eX_1e,\ldots, eX_ne)$.  Thus the second item implies the third.
\end{proof}

Notice that it follows from Theorem~\ref{t:Hpointlikes} and Theorem~\ref{thm:main} that $T\in \bHbar$ if and only if it satisfies all the pseudoidentities of the form $u'=(u')^2$, where we retain the notation of Theorem~\ref{t:Hpointlikes}, as was shown long ago in~\cite{AV03}.

For example, if $\bH$ is the trivial variety of groups, then we can take $E$ to consist of the identity $x=1$.  Then $e_x=x^{\omega}$ and so $x' = x^{\omega}xx^{\omega}= x^{\omega}x=:x^{\omega+1}$.  Thus the $\bHbar$-saturated subsemigroups of $2^T$ are the downclosed subsemigroups closed under $X\mapsto X^{\omega}\bigcup_{n\geq 1}X^n$, as was shown in~\cite{HRS2010AP}. If $\mathbf{Ab}$ denotes the variety of finite abelian groups, then we can take $E$ to  consist of the identity $xyx^{-1}y^{-1}=1$.  Thus a downclosed subsemigroup $S$ of $2^T$ is $\overline{\mathbf{Ab}}$-saturated if and only if whenever $X,Y$ belong to a subgroup of $S$, then $\bigcup_{n\geq 1} (XYX^{\omega -1}Y^{\omega-1})^n$ belongs to $S$ where, as usual, $x^{\omega-1}$ denotes the inverse of $x^{\omega+1}$ in the unique maximal subgroup of  $\widehat{\{x\}^+}$.  There is a single pseudoidentity $u=1$ in two variables that defines the variety $\mathbf G_{\mathrm{sol}}$ of finite solvable groups~\cite{BGGKPP06}.  Namely, let $u_1=x^{-2}y^{-1}x$ and $u_{n+1}= [xu_nx^{-1},yu_ny^{-1}]$, for $n\geq 0$. Then $u_n\to u$ with $u=1$ defining the variety of finite solvable groups.  Thus a downclosed subsemigroup $S$ of $2^T$ is $\overline{\mathbf G_{\mathrm{sol}}}$-saturated if and only if $\bigcup_{n\geq 1}u'(X,Y)^n$ belongs to $S$ for all $X,Y\in S$.  One can choose $e_u$ to be polynomial time computable~\cite{AV03}, so that $u'$ is polynomial time computable.  This will then give a faster approach to computing the $\overline{\mathbf G_{\mathrm{sol}}}$-pointlikes than working with $\mathbf G_{\mathrm{sol}}$-kernels.

The reader is referred to~\cite{Almeida:book,RS2009} for the notion of implicit operations and~\cite{AS2000} for the notion of implicit signatures.   Let $T$ be an $A$-generated finite semigroup.  Let $\gamma\colon \widehat{A^+}\to T$ be the canonical surjection and consider the relational morphism $\varphi_{\mathbf V}=\gamma^{-1} \pi_{\mathbf V}\colon T\to \widehat{F}_{\mathbf V}(A)$.  It is well known, cf.~\cite{RS2009}, that $X\subseteq T$ is $\mathbf V$-pointlike if and only if $X\subseteq \tau\varphi_{\mathbf V}^{-1}$ for some $\tau\in \widehat{F}_{\mathbf V}(A)$.  If $\sigma$ is an implicit signature containing multiplication, then a variety $\V$ of finite semigroups is \emph{weakly $\sigma$-reducible} for pointlikes if, for each finite $A$-generated semigroup $T$ and each $\V$-pointlike subset $X$ of $T$, there is a $\sigma$-term $\tau$ in $\widehat{F}_{\mathbf V}(A)$ with $X\subseteq \tau\varphi_{\mathbf V}^{-1}$.

\begin{lemma}\label{l:power.set.term}
Let $T$ be a finite semigroup and $S$ a profinite semigroup. Suppose that $\varphi\colon T\to S$ is a relational morphism whose graph $\#\varphi$ is closed in $T\times S$.   Let $u\in \widehat{A^+}$.  Then $u(s_1\varphi^{-1},\ldots, s_k\varphi^{-1})\subseteq u(s_1,\ldots s_k)\varphi^{-1}$ for any implicit operation $u$.
\end{lemma}
\begin{proof}
Write $u=\lim_{n\to \infty}w_n$ with the $w_n$ words.  By passing to a subsequence, we may assume that $w_n(s_1\varphi^{-1},\ldots, s_k\varphi^{-1})=u(s_1\varphi^{-1},\ldots, s_k\varphi^{-1})$ for all  $n\geq 1$ by finiteness of $2^T$.  Thus, for each $n\geq 1$, we have  $u(s_1\varphi^{-1},\ldots, s_k\varphi^{-1})=w_n(s_1\varphi^{-1},\ldots, s_k\varphi^{-1})\subseteq w_n(s_1,\ldots, s_k)\varphi^{-1}$ by definition of a relational morphism.  Therefore, if $t\in u(s_1\varphi^{-1},\ldots, s_k\varphi^{-1})$, then   $(t,w_n(s_1,\ldots, s_k))\in \#\varphi$ for each $n\geq 1$.  As $\#\varphi$ is closed, we may deduce that $(t,u(s_1,\ldots, s_k))\in \#\varphi$ and so $t\in u(s_1,\ldots, s_k)\varphi^{-1}$.  Thus $u(s_1\varphi^{-1},\ldots, s_k\varphi^{-1})\subseteq u(s_1,\ldots s_k)\varphi^{-1}$, as required.
\end{proof}

An implicit operation over $A$ is computable if there is a Turing machine that can compute its value given as input an $A$-tuple of elements of a finite semigroup $S$ (and the multiplication table of $S$).  We say that $\sigma$ is \emph{highly computable} if it consists of a recursively enumerable set of computable implicit operations.

The following generalizes one of the main results of~\cite{ACZ17}, but using weak reducibility instead of reducibility.  We shall use in the proof that if $u\in \widehat{A^+}$ with $|A|=n$, $T$ is a finite semigroup and $X_i\subseteq Y_i\subseteq T$, for $1\leq i\leq n$, then $u(X_1,\ldots, X_n)\subseteq u(Y_1,\ldots, Y_n)$ as can be seen by choosing a word agreeing with $u$ in the finite semigroup $2^T$.

\begin{theorem}\label{t:reducibility}
Let $\bH$ be a variety of finite groups defined by a set of group pseudoidentities $E$.  Let $\sigma$ consist of multiplication and the $u'$ with $u$ a left hand side of an element of $E$ and $u'$ defined as in Theorem~\ref{t:Hpointlikes}.  Then $\bHbar$-pointlikes are weakly $\sigma$-reducible.  If $E$ consists of a recursively enumerable basis of computable pseudoidentities, then $\sigma$ can be chosen to be highly computable.
\end{theorem}
\begin{proof}
Let $T$ be a finite $A$-generated subsemigroup and let $\varphi_{\bHbar}\colon T\to \widehat{F}_{\bHbar}(A)$ be the canonical relational morphism.  Let $S$ consist of all subsets $X$ of $T$ for which there is a $\sigma$-term $\tau\in \widehat{F}_{\bHbar}(A)$ with $X\subseteq \tau\varphi_{\bHbar}^{-1}$.  We claim that $S$ is an $\bHbar$-saturated subsemigroup of $2^T$ containing $T\eta$.  First of all, each singleton subset $\{t\}$ belongs to $S$ because we can take $\tau$ to be a word representing $t$. Clearly, $S$ is downclosed.   If $X,Y\in S$ with $X\subseteq \tau\varphi_{\bHbar}^{-1}$, $Y\subseteq \nu\varphi_{\bHbar}^{-1}$ with $\tau,\nu$ $\sigma$-terms, then $XY\subseteq \tau\nu\varphi_{\bHbar}^{-1}$.  So $S$ is a subsemigroup.  Suppose that $u=1$ belongs to $E$ with $n$ variables and $X_1,\ldots, X_n\in S$ with $X_i\subseteq \tau_i\varphi_{\bHbar}^{-1}$ with $\tau_i$ a $\sigma$-term.  Then $u'(\tau_1,\ldots, \tau_n)$ is a $\sigma$-term and $u'(X_1,\ldots, X_n)\subseteq u'(\tau_1,\ldots, \tau_n)\varphi_{\bHbar}^{-1}$ by Lemma~\ref{l:power.set.term}.   But $\bHbar\models u'=(u')^2$, as was observed by Almeida and Volkov~\cite{AV03}. Therefore, $u'(X_1,\ldots, X_n)^k\subseteq u'(\tau_1,\ldots,\tau_n)\varphi_{\bHbar}^{-1}$ for all $k\geq 1$ and so $\bigcup \langle u'(X_1,\ldots, X_n)\rangle \subseteq u'(\tau_1,\ldots,\tau_n)\varphi_{\bHbar}^{-1}$.   Therefore, $S$ is $\bHbar$-saturated by Theorem~\ref{t:Hpointlikes}.  We conclude that $S$ contains every $\bHbar$-pointlike set and so $\bHbar$ is weakly $\sigma$-reducible.

The final statement follows because there is an algorithm, which given a finite set $A$, produces a polynomial time computable idempotent $e_A$ is the minimal ideal of $\widehat{A^+}$~\cite{RZ00,AV03}.  Hence if $E$ is recursively enumerable and consists of computable implicit operations, then $\sigma$ can be taken to be highly computable.
\end{proof}

Examples where Theorem~\ref{t:reducibility} applies to obtain a highly computable implicit signature include $\bH$ any of the trivial variety, finite abelian groups, finite $p$-groups ($p$ prime), finite $\pi$-groups (with $\pi$ a recursive set of primes), finite nilpotent groups and finite solvable groups.  Each of these varieties can be defined by a single computable pseudoidentity in one or two variables (computable in polynomial time except for in the case of $\pi$-groups, where the complexity depends on the membership algorithm for $\pi$).

\def\cprime{$'$} \def\cprime{$'$}
\providecommand{\bysame}{\leavevmode\hbox to3em{\hrulefill}\thinspace}
\providecommand{\MR}{\relax\ifhmode\unskip\space\fi MR }
\providecommand{\MRhref}[2]{%
  \href{http://www.ams.org/mathscinet-getitem?mr=#1}{#2}
}
\providecommand{\href}[2]{#2}


\begin{thebibliography}{10}

\bibitem{K17}
K.~{Alibabaei}, \emph{{The pseudovariety of all nilpotent groups is tame}},
  ArXiv e-prints (2017), 1712.09547.

\bibitem{Almeida:book}
J.~Almeida, \emph{Finite semigroups and universal algebra}, Series in Algebra,
  vol.~3, World Scientific Publishing Co. Inc., River Edge, NJ, 1994,
  Translated from the 1992 Portuguese original and revised by the author.

\bibitem{Almeida99}
\bysame, \emph{Some algorithmic problems for pseudovarieties}, Publ. Math.
  Debrecen \textbf{54} (1999), no.~suppl., 531--552, Automata and formal
  languages, VIII (Salg\'otarj\'an, 1996).

\bibitem{Jvg}
J.~Almeida, A.~Azevedo, and M.~Zeitoun, \emph{Pseudovariety joins involving
  {$\mathscr J$}-trivial semigroups}, Internat. J. Algebra Comput. \textbf{9}
  (1999), no.~1, 99--112.

\bibitem{ACZ08}
J.~Almeida, J.~C. Costa, and M.~Zeitoun, \emph{Pointlike sets with respect to
  {${\bf R}$} and {${\bf J}$}}, J. Pure Appl. Algebra \textbf{212} (2008),
  no.~3, 486--499. 

\bibitem{ACZ17}
\bysame, \emph{Reducibility of pointlike problems}, Semigroup Forum \textbf{94}
  (2017), no.~2, 325--335.

\bibitem{AS2000}
J.~Almeida and B.~Steinberg, \emph{On the decidability of iterated semidirect
  products with applications to complexity}, Proc. London Math. Soc. (3)
  \textbf{80} (2000), no.~1, 50--74.

\bibitem{AV03}
J.~Almeida and M.~V. Volkov, \emph{Profinite identities for finite semigroups
  whose subgroups belong to a given pseudovariety}, J. Algebra Appl. \textbf{2}
  (2003), no.~2, 137--163.

\bibitem{Jhyp}
J.~Almeida and M.~Zeitoun, \emph{The pseudovariety {${\bf J}$} is
  hyperdecidable}, RAIRO Inform. Th{\'e}or. Appl. \textbf{31} (1997), no.~5,
  457--482.

\bibitem{Ash91}
C.~J. Ash, \emph{Inevitable graphs: a proof of the type {${\rm II}$} conjecture
  and some related decision procedures}, Internat. J. Algebra Comput.
  \textbf{1} (1991), no.~1, 127--146.

\bibitem{AS03}
K.~Auinger and B.~Steinberg, \emph{On the extension problem for partial
  permutations}, Proc. Amer. Math. Soc. \textbf{131} (2003), no.~9, 2693--2703.

\bibitem{BGGKPP06}
T.~Bandman, G.-M. Greuel, F.~Grunewald, B.~Kunyavski\accentu\i, G.~Pfister, and
  E.~Plotkin, \emph{Identities for finite solvable groups and equations in
  finite simple groups}, Compos. Math. \textbf{142} (2006), no.~3, 734--764.

\bibitem{DMS07}
M.~Delgado, A.~Masuda, and B.~Steinberg, \emph{Solving systems of equations
  modulo pseudovarieties of abelian groups and hyperdecidability}, Proceedings
  of the International Conference Semigroups and Formal Languages
  (J.~Andr{\'e}, V.~H. Fernandes, M.~J.~J. Branco, GMS Gomes, J.~Fountain, and
  J.~C. Meakin, eds.), World Scientific, 2007, pp.~57--65.

\bibitem{Eilenberg76}
S.~Eilenberg, \emph{Automata, languages, and machines. {V}ol. {B}}, Academic
  Press, New York, 1976, With two chapters by Bret Tilson, Pure and Applied
  Mathematics, Vol. 59.

\bibitem{GooSte2017merge}
S.~J.~v. {Gool} and B.~{Steinberg}, \emph{{Merge decompositions, two-sided
  {K}rohn-{R}hodes, and aperiodic pointlikes}}, ArXiv e-prints (2017),
  1708.08118.

\bibitem{GriNekSus2000}
R.~I. Grigorchuk, V.~V. Nekrashevich, and V.~I. Sushchanski\accentu\i,
  \emph{Automata, dynamical systems, and groups}, Tr. Mat. Inst. Steklova
  \textbf{231} (2000), 134--214.

\bibitem{Hen1988}
K.~Henckell, \emph{Pointlike sets: the finest aperiodic cover of a finite
  semigroup}, J. Pure Appl. Algebra \textbf{55} (1988), 85--126.

\bibitem{Henckell04}
\bysame, \emph{Idempotent pointlike sets}, Internat. J. Algebra Comput.
  \textbf{14} (2004), no.~5-6, 703--717.

\bibitem{HMPR}
K.~Henckell, S.~W. Margolis, J.-E. Pin, and J.~Rhodes, \emph{Ash's type {${\rm
  II}$} theorem, profinite topology and {M}al{\cprime}cev products. {I}},
  Internat. J. Algebra Comput. \textbf{1} (1991), no.~4, 411--436.

\bibitem{HR91}
K.~Henckell and J.~Rhodes, \emph{The theorem of {K}nast, the {$PG=BG$} and
  type-{${\rm II}$} conjectures}, Monoids and semigroups with applications
  (Berkeley, CA, 1989), World Sci. Publ., River Edge, NJ, 1991, pp.~453--463.

\bibitem{HRS2010AP}
K.~Henckell, J.~Rhodes, and B.~Steinberg, \emph{Aperiodic pointlikes and
  beyond}, Internat. J. Algebra Comput. \textbf{20} (2010), no.~2, 287--305.

\bibitem{HRS2010}
\bysame, \emph{A profinite approach to stable pairs}, Internat. J. Algebra
  Comput. \textbf{20} (2010), no.~2, 269--285.

\bibitem{HRS2012}
\bysame, \emph{An effective lower bound for group complexity of finite
  semigroups and automata}, Trans. Amer. Math. Soc. \textbf{364} (2012), no.~4,
  1815--1857.

\bibitem{KR65}
K.~Krohn and J.~Rhodes, \emph{Algebraic theory of machines. {I}. {P}rime
  decomposition theorem for finite semigroups and machines}, Trans. Amer. Math.
  Soc. \textbf{116} (1965), 450--464.

\bibitem{KR68}
\bysame, \emph{Complexity of finite semigroups}, Ann. of Math. (2) \textbf{88}
  (1968), 128--160.

\bibitem{PlaZei2016FO}
T.~Place and M.~Zeitoun, \emph{Separating regular languages with first-order
  logic}, Logical Methods in Computer Science \textbf{12} (2016), no.~1:5,
  1--30.

\bibitem{RZ00}
N.~R. Reilly and S.~Zhang, \emph{Decomposition of the lattice of
  pseudovarieties of finite semigroups induced by bands}, Algebra Universalis
  \textbf{44} (2000), no.~3-4, 217--239.

\bibitem{RS99}
J.~Rhodes and B.~Steinberg, \emph{Pointlike sets, hyperdecidability and the
  identity problem for finite semigroups}, Internat. J. Algebra Comput.
  \textbf{9} (1999), no.~3-4, 475--481.

\bibitem{RS2009}
\bysame, \emph{{The q-theory of Finite Semigroups}}, Springer, 2009.

\bibitem{Slice}
B.~Steinberg, \emph{On pointlike sets and joins of pseudovarieties}, Internat.
  J. Algebra Comput. \textbf{8} (1998), no.~2, 203--234.

\bibitem{St01}
\bysame, \emph{Inevitable graphs and profinite topologies: some solutions to
  algorithmic problems in monoid and automata theory, stemming from group
  theory}, Internat. J. Algebra Comput. \textbf{11} (2001), no.~1, 25--71.

\bibitem{Slice2}
\bysame, \emph{On algorithmic problems for joins of pseudovarieties}, Semigroup
  Forum \textbf{62} (2001), no.~1, 1--40.

\bibitem{Str79}
H.~Straubing, \emph{Aperiodic homomorphisms and the concatenation product of
  recognizable sets}, J. Pure Appl. Algebra \textbf{15} (1979), no.~3,
  319--327.

\bibitem{Straubingdelay}
\bysame, \emph{Finite semigroup varieties of the form {${\bf V}\ast {\bf D}$}},
  J. Pure Appl. Algebra \textbf{36} (1985), no.~1, 53--94.

\bibitem{Str1994}
\bysame, \emph{Finite automata, formal logic, and circuit complexity}, Progress
  in Theoretical Computer Science, Birkh\"auser Boston Inc., Boston, 1994.

\bibitem{Tilson}
B.~Tilson, \emph{Categories as algebra: an essential ingredient in the theory
  of monoids}, J. Pure Appl. Algebra \textbf{48} (1987), no.~1-2, 83--198.

\bibitem{Zeiger1}
P.~Zeiger, \emph{Yet another proof of the cascade decomposition theorem for
  finite automata}, Math. Systems Theory \textbf{1} (1967), no.~3, 225--228.

\end{thebibliography}
\end{document}